\documentclass{article}
\usepackage[a4paper, total={6in, 8in}]{geometry}

\usepackage{indentfirst}

\usepackage{amsthm}
\usepackage{amssymb}
\usepackage{amsmath}

\usepackage{pdfpages}
\usepackage{hyperref}

\usepackage{graphicx}
\usepackage{tikz-cd}
\usepackage{stmaryrd} 

\newtheorem{theorem}{Theorem}[section]
\newtheorem{lemma}[theorem]{Lemma}
\newtheorem{corollary}[theorem]{Corollary}
\newtheorem{proposition}[theorem]{Proposition}
\newtheorem{remark}[theorem]{Remark}
\theoremstyle{definition}
\newtheorem{definition}[theorem]{Definition}

\newcommand{\Z}{\mathbb{Z}}
\newcommand{\R}{\mathbb{R}}
\newcommand{\C}{\mathbb{C}}

\newcommand{\RP}{\mathbb{RP}}
\newcommand{\Sphere}{\mathbb{S}}
\newcommand{\Isom}{\operatorname{Isom}}

\newcommand{\Mat}{\operatorname{Mat}}
\newcommand{\GL}{\operatorname{GL}}
\newcommand{\SL}{\operatorname{SL}}
\newcommand{\Ortho}{\operatorname{O}}

\newcommand{\PGL}{\operatorname{PGL}}

\newcommand{\Hom}{\operatorname{Hom}}
\newcommand{\Id}{\operatorname{Id}}
\newcommand{\orb}{\mathcal{O}}

\newcommand{\rom}[1]{\uppercase\expandafter{\romannumeral#1}}

\newcommand{\Def}{\operatorname{Def}}
\newcommand{\Defss}{\operatorname{Def}^{\operatorname{ss}}}
\newcommand{\Defgen}{\operatorname{Def}_{\operatorname{gen}}}
\newcommand{\Defcon}{\operatorname{Def}_{\operatorname{con}}}
\newcommand{\Defssgen}{\operatorname{Def}^{\operatorname{ss}}_{\operatorname{gen}}}
\newcommand{\Defsscon}{\operatorname{Def}^{\operatorname{ss}}_{\operatorname{con}}}

\newcommand{\diag}{\operatorname{diag}}
\newcommand{\Ann}{\operatorname{Ann}}

\title{Convex real projective structures on Coxeter orbifold $D^2(;n_1,n_2,n_3,n_4)\times \mathbb{R}$}
\author{Jaesung Bae}
\date{}

\begin{document}

\maketitle

\begin{abstract}
   The deformation space of real projective structures parametrizes the space of the convex real projective structures on an orbifold. The Coxeter orbifold can be obtained $D^2(;n_1,n_2,n_3,n_4)\times\mathbb{R}$ by embedding the Coxeter quadrilateral from $\mathbb{RP}^2$ to $\mathbb{RP}^3$, and extending the side edges to planes and perturbing these to obtain a convex polytope. This noncompact orbifold can be induced from two combinatorial polytopes. Using this fact, we determine the deformation space of real projective structures on this orbifold with parametrization, and see its very detailed properties.
\end{abstract}

\textbf{Keywords} : 
Convex real projective structure, Coxeter orbifold, Deformation space, Character variety

\tableofcontents

\section[Introduction]{Introduction} \label{ch:1} \noindent

Understanding moduli spaces of geometric structures is a fundamental problem in hyperbolic geometry. However, in dimensions three or higher, Mostow rigidity implies that the deformation space of hyperbolic structures is trivial. This motivates us to explore alternative geometric structures, such as real projective structures, to gain deeper insights into deformation spaces.

Let $\Sphere^n=\R^{n+1}-\{0\}/\sim$ be the 
\textit{real projective $n$-sphere} consisting of the open rays of $\R^{n+1}$ from the origin. The group $\SL_\pm(n+1,\R)$ can be viewed as the group of real projective automorphisms of $\Sphere^n$. A \textit{real projective structure} on a smooth $n$-sphere $\orb$ is an atlas of coordinate charts from $\orb$ to $\Sphere^n$ such that the changes of coordinates locally lie in $\SL_\pm(n+1,\R)$. A polytope equipped with a real projective structure is called a \textit{projective polytope}. We can consider a \textit{projective reflection}, whose kernel is a hyperplane containing each side of the projective polytope in the vector space $V$. By silvering the sides and removing some vertices or edges of the projective polytope, we can construct an orbifold. Specifically, we focus on a special type of projective polytope and its orbifold, whose orbifold fundamental group is a well-behaved group, called the \textit{Coxeter group}. A \textit{Coxeter polytope} is a projective polytope whose projective reflections generate a Coxeter group. We call its induced orbifold a \textit{Coxeter orbifold}. The \textit{deformation space} is the space of real projective structures on the orbifold modulo isotopy equivalence.

In general, it is challenging to compute the deformation space directly. However, the deformation space of some Coxeter orbifolds can be computed by examining the relationship between the deformation space and other objects. We provide an injective open map from the deformation space to the \textit{character variety}, which can be computed from the representations of the group homomorphisms.

A $(G, X)$-structure on a Coxeter polytope can be defined in terms of the projective reflections of its sides. More precisely, it is defined by a linear functional and its eigenvector for each projective reflection. Hence, we can compute the space of representations in $\Hom(\Gamma, G)$, where $\Gamma$ is the orbifold fundamental group of the Coxeter orbifold, which is identical to the related Coxeter group. To ensure that the induced Coxeter orbifold has a nonempty interior, these reflections must satisfy specific algebraic conditions. One of the conditions involves the dihedral angles, which correspond to the orders of two generators of the Coxeter group. This order is called an \textit{edge order}. (The name edge comes from the Coxeter diagram since an edge of the diagram corresponds to the dihedral angle between two sides, which are the nodes of the diagram.)

Choi and Goldman \cite{CG05} studied the deformation spaces of Coxeter 2-orbifolds. Thus, the deformation spaces of Coxeter 3-orbifolds or higher-dimensional Coxeter orbifolds are interesting problems. Choi and Lee \cite{CL15} determined the deformation spaces of orderable Coxeter orbifolds, and Marquis \cite{Mar10} studied the deformation spaces of truncation polytopes. The underlying spaces of these orbifolds are compact spaces with certain isolated points removed. In this paper, we will determine the deformation space of a specific orbifold whose underlying space is a compact space with certain edges removed.

We will compute the group homomorphisms and the character variety to understand the deformation space of a specific Coxeter orbifold: the product of a topological disk with four finite corner reflectors and a real line, i.e., $D^2(; n_1, n_2, n_3, n_4) \times \R$. Additionally, we will verify Porti’s theorem \cite{Por23}, which describes a local homeomorphism from the character variety for $\SL_\pm(n,\R)$ to $\SL_\pm(n+1,\R)$ for the same orbifold. Finally, we will examine how the conditions of convex cocompactness in the representations affect the deformation space.

We are also interested in determining which Coxeter polytope can induce a given Coxeter orbifold. When all edge orders are finite, the Coxeter orbifold can be induced by removing only a few points from the corresponding Coxeter polytope, resulting in a unique Coxeter polytope up to combinatorial equivalence. However, in our case, involving infinite edge orders, the same Coxeter orbifold can be induced by combinatorially distinct polytopes. Therefore, we aim to distinguish such polytopes that induce the convex Coxeter $D^2(; n_1, n_2, n_3, n_4) \times \R$ through the following remark.

\begin{remark}\label{Remark1}
The Coxeter polytope $P$, whose Coxeter orbifold is $D^2(; n_1, n_2, n_3, n_4) \times \R$, with the assumption that all $n_j$ are finite and greater than or equal to 3, can be divided into two cases, since both satisfy the conditions of the same Coxeter group. Here, $\alpha_j$ represents the linear functional of projective reflection $R_j$ for each $j=1, 2, 3, 4$.

\begin{itemize}
\item [1.] If $\alpha_j$'s are linearly independent, then $P$ is a Coxeter 3-simplex where two non-adjacent edge orders are infinite, and the other edge orders are identical to $n_1, n_2, n_3$, and $n_4$.
\item [2.] If $\alpha_j$'s are linearly dependent, then $P$ is $\orb^2_4 \times \R$, where $\orb^2_4$ is a Coxeter quadrilateral with edge orders $n_1, n_2, n_3, n_4$, and $\R$ is orthogonally producted to $\orb^2_4$.
\end{itemize}
\end{remark}

We refer to the first case as the \textit{general position case} and the second case as the \textit{concurrent case}. We introduce some notations for simplicity.
Let $\Def(\hat{P})$ denote the deformation space of $\hat{P}$, and let $\Defss(\hat{P})$ denote the subset of $\Def(\hat{P})$ associated with semisimple representations. (See Section \ref{ch:2.2}.) Define $\Defgen(\hat{P})$ (resp. $\Defssgen(\hat{P})$) as the subspace of $\Def(\hat{P})$ (resp. $\Defss(\hat{P})$) corresponding to the general position case, and define $\Defcon(\hat{P})$ (resp. $\Defsscon(\hat{P})$) as the subspace of $\Def(\hat{P})$ (resp. $\Defss(\hat{P})$) corresponding to the concurrent case. (See Chapter \ref{ch:3}) Now, we present the main theorem of this paper.

\begin{theorem}[Main Theorem]\label{Main Theorem}
Let $\hat{P}$ be a convex Coxeter orbifold $D^2(; n_1, n_2, n_3, n_4) \times \R$, where $n_1, n_2, n_3, n_4$ are positive integers greater than or equal to 3.  Then,

\begin{itemize}
\item $\Defss(\hat{P})$ is homeomorphic to $\R^3 \times [4, \infty)^2$.
\item $\Defsscon(\hat{P})$ is homeomorphic to $\R^4$.
\item $\Defssgen(\hat{P})$ has two connected components.
\item $\Defgen(\hat{P})$ is homeomorphic to $\R^3 \times [4, \infty)^2$.
\item $\Def(\hat{P})$ has two non-Hausdorff components.
\end{itemize}
\end{theorem}

We make several contributions to the study of deformation spaces of Coxeter orbifolds.  
First, we compute the deformation space of a specific Coxeter orbifold, $D^2(; n_1, n_2, n_3, n_4) \times \R$, and provide an explicit parametrization.  
Second, we analyze the effects of convex cocompactness in the representation space and its implications for the deformation space.  
Third, we give relations between deformation spaces and character varieties of orbifolds whose underlying spaces are compact spaces with edges removed as well as isolated points removed. 
Finally, we distinguish different Coxeter polytopes that induce the same Coxeter orbifold when edge orders are infinite, revealing structural differences between these polytopes.

\subsection{Outline of the paper}

In chapter \ref{ch:2}, we will see fundamental definitions of Coxeter orbifolds, deformation spaces, and character varieties.
In chapter \ref{ch:3}, we will see the proof of the main theorem and analyze key arguments step by step.
We will examine Porti's theorem and analyze additional properties of the deformation space through convex cocompactness in chapter \ref{ch:4}. (See Proposition \ref{Additional prop}.)
We will discusses additional computations and numerical experiments using Mathematica in chapter \ref{ch:5}.

\section[Preliminaries]{Preliminaries}
\label{ch:2}

\subsection{Coxeter orbifolds}
\label{ch:2.1}

\subsubsection{Basic definitions of orbifolds}
\label{ch:2.1.1}

Recall the definition of a manifold. A manifold is a topological space which is locally Euclidean, second countable (or equivalently, paracompact), and Hausdorff. More rigorously, a manifold is defined via a collection of manifold charts, with its structure determined by a maximal atlas consisting of compatible local homeomorphisms to open subset of Euclidean space. An orbifold can be regarded as a generalization of this framework, wherein the local charts are allowed to model neighborhoods as quotients of Euclidean spaces by finite group actions. 

\begin{definition} \label{Def:orb}
Let $|\orb|$ be a Hausdorff topological space. An \textit{orbifold chart} is a collection of triples $(\tilde{U}_\alpha, G_\alpha, \phi_\alpha)$ for each index $\alpha$ satisfying the followings.
\begin{itemize}
    \item [(1)] $\tilde{U}_\alpha$ is a connected open subset in $\R^n$.
    \item [(2)] $G_\alpha$ is a finite group acting smoothly on $\tilde{U}_\alpha$.
    \item [(3)] $\phi_\alpha : \tilde{U}_\alpha \rightarrow U_\alpha$ is a $G_\alpha$-invariant map, which defines a homeomorphism between $\tilde{U}_\alpha / G_\alpha$ and $U_\alpha$ where $U_\alpha \subset |\orb|$ is an open subset.
    \item [(4)] The collection of $U_\alpha$ covers $|\orb|$. 
\end{itemize}

An \textit{orbifold atlas} is a family of compatible orbifold charts. An \textit{orbifold} $\orb$ is a topological space $X$ with a maximal orbifold atlas. We denote $|\orb|$ its \textit{underlying space}.
    
\end{definition}

To define the orbifold version of the fundamental group, define the orbifold version of path, called $\mathcal{G}$-path. Let $\mathcal{G}$ be a groupoid of local diffeomorphisms of $\orb$. Take a subdivision of $[a,b]$ with $a=t_0 \leq t_1 \leq \cdots \leq t_k = b$. A \textit{$\mathcal{G}$-path} is a map $r : [a,b] \rightarrow \orb$ which is continuous on each $[t_{i-1},t_i]$ and an element $g_i \in \mathcal{G}$ is defined in a neighborhood of $c_i(t_i)$, for each $i$, such that $g_i c_i(t_i) = c_{i+1}(t_i)$, where $c_i$ is the restriction of $r$ on $[t_{i-1},t_i]$. In the point of covering space of $\orb$, $\mathcal{G}$-path is a continuous path allowing movement on to other preimages of covering map with the same image point for finite times. A \textit{$\mathcal{G}$-loop} is a $\mathcal{G}$-path with same endpoints. Refer Figure \ref{CJ22 Figure}. The \textit{orbifold fundamental group} based at $x_0 \in \orb$ consists of homotopy classes of $\mathcal{G}$-loops based at $x_0 \in \orb$, denote $\pi_1^{\operatorname{orb}}(\orb,x_0)$.

\begin{figure}[h]
    \centering
    \includegraphics[height=4cm]{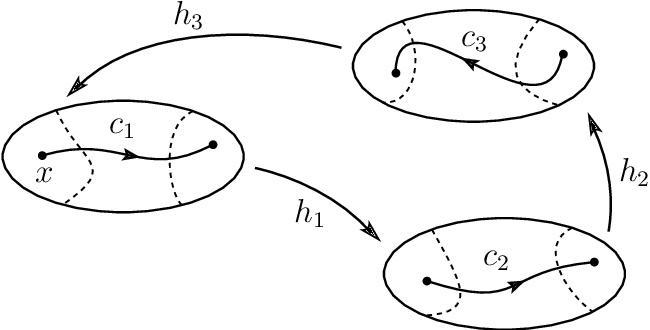}
    \includegraphics[height=4cm]{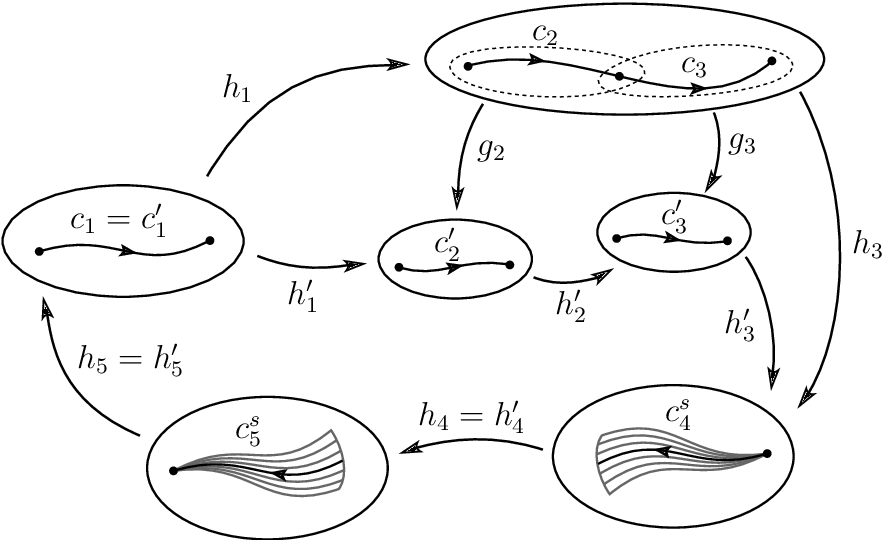}
    \caption{A $\mathcal{G}$-loop and homotopic $\mathcal{G}$-loops (\cite{CJ22}, Figure 2.1 \& 2.2)}
    \label{CJ22 Figure}
\end{figure}

Recall that we can dismiss the notation of base point on the fundamental group of a manifold if the given manifold is connected. Similarly, if $\orb$ is a connected orbifold, then we can simplify the notation, $\pi_1^{\operatorname{orb}}(\orb,x_0)=\pi_1^{\operatorname{orb}}(\orb)$. Since we consider only orbifold fundamental groups throughout this paper, we write $\pi_1(\orb)=\pi_1^{\operatorname{orb}}(\orb)$ for simplicity.

The definition of  \textit{Euler characteristic of orbifold} is $\chi^{\operatorname{orb}}(\orb)=\sum(-1)^{\dim s_k} \frac{1}{N_{s_k}}$ where $s_k$ is open $k$-cell in the triangulation and $N_{s_k}$ is the order of the local group. Instead of the original definition, \textit{Riemann-Hurwitz formula} is useful to compute the Euler characteristic of orbifold.
\[
\chi^{\operatorname{orb}}(\orb) = \chi(|\orb|)-\sum_{i=1}^{m}\left( 1-\frac{1}{q_i} \right) - \frac{1}{2} \sum_{j=1}^{n} \left( 1-\frac{1}{r_j} \right) - \frac{1}{2} n_\orb
\]
Here, $q_i$ is the order of each cone point, $r_j$ is the order of each corner-reflection, and $n_\orb$ is the number of full 1-orbifold boundary components. In this paper, we will simply denote $\chi(\orb)=\chi^{\operatorname{orb}}(\orb)$.

We define the orientation of an orbifold as the following. Let $\orb$ be an orbifold. $\orb$ is \textit{locally orientable} if there is an atlas with orbifold charts $(\tilde{U}_\alpha,G_\alpha,\phi_\alpha)$ such that each $G_\alpha$ acts by a diffeomorphism of $\tilde{U}_\alpha$ which preserves the orientation. If there is an orientation such that every embedding between charts is orientation-preserving for each $\tilde{U}_\alpha$, then we say that $\orb$ is \textit{orientable}.

We can classify the singular points of 2-orbifolds. Let $\orb^2$ be a 2-orbifold. A \textit{mirror point} is a point having an orbifold chart with $\R^2/\Z_2$, where $\Z_2$ is generated by the orthogonal reflection on the $y$-axis. A \textit{cone-point} of order $n$ is a point having an orbifold chart with $\R^2/\Z_n$, where $\Z_n$ acts by rotations by angles $2\pi \frac{m}{n}$ for integers $m$. A \textit{corner-reflector} of order $n$ is a point having an orbifold chart with $\R^n/D_n$, where $D_n$ is the dihedral group of order $2n$ which is generated by reflections about two lines meeting at an angle $\frac{\pi}{n}$. We denote $D^2(m_1,\cdots,m_i ; n_1,\cdots,n_j)$ the orbifold where $m_{i'}$ is a cone point of order $m_{i'}$ and $n_{j'}$ is a corner-reflector of order $n_{j'}$ whose underlying space is a topological disk. In this paper, we will focus on $D^2(;n_1,n_2,n_3,n_4)$, the orbifold having four corner-reflectors whose underlying space is homeomorphic to a disk. A \textit{silvering} can be defined on mirror points similar to the definition on a manifold with boundary. The silvering can be generalized on $n$-orbifold with $n\geq3$, where the orbifold chart of the singular point is generated by $\Z_2$ with orthogonal reflection.

Refer to Choi's book to see more details about an orbifold. (\cite{Cho12}, chapter 4)

\subsubsection{Convex real projective structure}
\label{ch:2.1.2}

Let $X$ be a geometric space and $G$ be an Lie group acting on $X$ transitively and effectively. Then, the \textit{$(G,X)$-structure for an $n$-orbifold} $\orb^n$ is a maximal atlas of charts to orbit spaces of finite subgroups of $G$ acting on open subsets of $X$. In other words, we can define $(G,X)$-structure for an orbifold from the definition \ref{Def:orb} by replacing $\R^n$ to $X$, and $G_\alpha$ to some finite subgroups of $G$.

We can get some geometric structures from geometries since the isometry group of $X$, denoted by $\Isom(X)$, is a Lie group acting on $X$ transitively and effectively. Here are some examples. A Euclidean structure is $(\Isom(\R^n),\R^n)$-structure with $\Isom(\R^n)= \Ortho(n) \ltimes \R^n$. A spherical structure is $(\Isom(S^n),S^n)$-structure where 
\[S^n=\{ x=(x_0,\cdots,x_n) \in \R^{n+1} \; | \; x_0^2 + x_1^2 + \cdots + x_n^{2} = 1\}\] and $\Isom(S^n) = \Ortho(n+1)$. A hyperbolic structure is $(\Isom(H^n),H^n)$ where \[H^n=\{ x=(x_0,\cdots,x_n) \in \R^{n+1} \; | \; -x_0^2 + x_1^2 + \cdots + x_n^{2} = 1, \; x_0>0 \}\]
and 
\[\Isom(H^n)=\Ortho(n,1)=\{ JA \; | \; J=diag(-1,1,\cdots,1), A \in \Ortho(n+1) \}.\]

On the other side, convex real projective structures does not come from some geometry. Instead, a Lie group $\PGL(n+1,\R)$ takes a roll of automorphism group of real projective space $\RP^n$. Recall that the projective space $\RP^n$ is the quotient space $\R^{n+1}-\{0\}/\sim$, where the equivalence relation $\sim$ is defined by
\[
x\sim y \quad \Leftrightarrow \quad x=ty \quad \text{ for } t\in \R-\{0\}
\]
, and the projective linear group $\PGL(n+1,\R)$ is $\GL(n+1,\R)/\sim$, where $\sim$ is defined by 
\[
A\sim B \quad \Leftrightarrow \quad A=tB \quad \text{ for } t\in \R-\{0\}.
\]
Moreover, convex real projective structures can be represented by the covering space of $\RP^n$. Let $\Sphere^n$ be the unit sphere which is the image of the \textit{spherical projection} $\Sphere : \R^{n+1}-\{0\} \rightarrow \Sphere^n $ which sends $x$ to an equivalence class $[x]$ where the equivalence relation is defined by
\[
x\sim y \quad \Leftrightarrow \quad x=ty \quad \text{ for } t>0.
\]
Let
\[
\SL_\pm(n+1,\R)=\{A \in GL(n+1,\R) \; | \; \det(A)=\pm1 \}.
\]
Then, we can note that $\PGL(n+1,\R)=\SL_\pm(n+1,\R)/\{\pm I\}$ and $\Sphere^n=\RP^n/\{\pm 1\}$. Therefore, we can conclude that $(\SL_\pm(n+1,\R),\Sphere^n)$-structures also possesses the convex real projective structures. 

\begin{definition} \label{def:projective structure}

A \textit{convex real projective structure} on an $n$-orbifold $\orb^n$ is a $(\PGL(n+1),\RP^n)$-structure (or $(\SL_\pm(n+1,\R),\Sphere^n)$-structure).

\end{definition}

It is easy to see that the isometry group of $E^n$, $S^n$, and $H^n$ can be represented by orthogonal groups. $\Isom(S^n)$ is an orthogonal group by itself, and $\Isom(H^n)$ can be represented with some diagonal matrix $J$. For Euclidean structures, there is a one-to-one correspondence between $O(n)\ltimes \R^{n}$ and the subset of $\Mat(n+1,\R)$, by
\[
\begin{bmatrix}
& & & |  \\
\multicolumn{3}{c}
{\scalebox{1.5}{$A$}} & \textit{b}  \\
& & & | \\
0 & \cdots & 0 & 1 
\end{bmatrix}
\]
where $A \in O(n)$ and $b \in \R^{n}$ as an $n$-by-1 matrix. (Refer to chapter 5.2 of \cite{Rat19}.) 
Not only for Lie groups, each underlying space of $E^n$, $S^n$, and $H^n$ can be defined by spherical projection of some subset of $\R^{n+1}$ from the origin. Therefore, we can conclude that either euclidean, spherical, or hyperbolic structure is also a convex real projective structure.

In this paper, we will use $(\SL_\pm(n+1,\R),\Sphere^n)$-structures to see convex real projective structures.

\subsubsection{Coxeter orbifolds}
\label{ch:2.1.3}

A \textit{Coxeter group} $\Gamma$ is a group defined by
\[
\Gamma = \langle R_i \; | \; R_i^2, (R_i R_j)^{n_{ij}} \rangle
\]
where $n_{ij} \in \Z_{\geq2} \cup \{\infty\}$ for every $i,j\in I$. Here, $I$ is a countable index set and $n_{ij}$ is symmetric, in other words, $n_{ij}=n_{ji}$.

We will define the projective polytope by the spherical projection. Let $V=\R^{n+1}$. A \textit{projective reflection} is an element of $\SL_\pm(V)$ of order 2 fixing hyperplane of $\Sphere(V)$ pointwise. In other words, a projective reflection is $R(x)=x-\alpha(x)v$ where $v \in V$ is an eigenvector with eigenvalue -1, and $\alpha$ is its linear functional on the dual space $V^*$ with $\alpha(v)=2$. For simplicity, $R=\Id-\alpha \otimes v$. A \textit{convex polyhedral cone} is $K=\bigcap_i\{x\in V \; | \; \alpha_i(x) \geq 0\}$ for a finite number of reflections $R_i = \Id -\alpha_i \otimes v_i$. A \textit{projective polytope} $P=\Sphere(K)$ is a spherical projection of $K$, which is a properly convex subset of $\Sphere(V)$ with nonempty interior.

\begin{definition} \label{def:Coxeter polytope}

A \textit{(projective) Coxeter polytope} is a projective polytope $P$ with the set of its sides $\mathcal{S}=\{S_i \; | \; i=1,\cdots,m\}$ and projective reflections $R_i = \Id - \alpha_i \otimes v_i$ such that each side $S_i$ is contained in $\Sphere(\ker(\alpha_i))$ and satisfying the following conditions. For any $S_i, S_j \in \mathcal{S}$,

\begin{itemize}
    \item[(C1)] $\alpha_i(v_j)=2$ $\Leftrightarrow$ $i=j$
    \item[(C2)] $\alpha_i(v_j)\leq 0$ for every $i\neq j$
    \item[(C3)] $\alpha_i(v_j)=0$ $\Rightarrow$ $\alpha_j(v_i)=0$
    \item[(C4)] $\alpha_i(v_j) \alpha_j(v_i) \geq 4$ or $\alpha_i(v_j) \alpha_j(v_i) = 4\cos^2(\frac{\pi}{n_{ij}})$ for some $n_{ij}\in \Z_{\geq2}$.
    \item[(C5)] The interior of the convex polyhedral cone $K^\circ$ is not empty.
\end{itemize}

In this paper, we will say \textit{Vinberg's conditions} for these five conditions.
    
\end{definition}

 Note that $\pi/n_{ij}$ in (C4) is exactly the same as the dihedral angle between $S_i$ and $S_j$. In addition, $\alpha_i(v_j) \alpha_j(v_i) \geq 4$ holds if and only if the dihedral angle is zero, which can be accepted since $n_{ij}$ is infinite. Therefore, we call this $n_{ij}\in \Z_{\geq2}\cup \{\infty\}$ an \textit{edge order}. We can easily see that a Coxeter polytope is determined by $(\alpha_1, \cdots, \alpha_f, v_1 \cdots, v_f) \in (V^*)^f \times V^f$ where $f$ is the number of sides of the Coxeter polytope. For simplicity, denote $(\alpha,v)$. To understand this tuple easily, we define the \textit{Cartan matrix} $M$ such that $(M)_{ij}=\alpha_i(v_j)$. To understand the given Coxeter polytope easily, we use a figure of the polytope whose ridges are labeled with dihedral angles, called a \textit{labeled polytope}.

Let $P$ be a Coxeter polytope whose every edge order is finite. We can obtain an orbifold by silvering the sides and removing some vertices of $P$. The condition of those vertices can be explained as the followings which are all equivalent.
\begin{itemize}
    \item The vertex of $P$ cannot induce an orbifold chart centered at the vertex.
    \item The action group $G$ of an orbifold chart $\tilde{U}/G \cong U$ centered at the vertex cannot be a finite group.
    \item From the labeled polytope, the vertex has more than three edges ending or has exactly three edges with labeling except
    \[
    (2,2,n), (2,3,3), (2,3,4), (2,3,5)
    \]
    for any $n \geq 2$.
\end{itemize}

In this paper, we also remove every edge whose order is infinite because we cannot find the finite action group to make an orbifold chart centered at such points. 

\begin{definition} \label{Coxeter orbifold}

Let $P$ be a Coxeter polytope, and $G$ be the group generated by projective reflections of $P$. For a point $x\in P$, let $G_x$ be a stabilizer group of $G$ with respect to $x$. Then, a \textit{Coxeter orbifold} $\hat{P}$ is an orbifold satisfying the followings.

\begin{itemize}
    \item[(1)] If $G_x$ is finite, $x\in|\hat{P}|$ and we have an orbifold chart centered at $x$ using $G_x$.
    \item[(2)] If $G_x$ is infinite, $x\notin |\hat{P}|$.
\end{itemize}
    
\end{definition}

Note that the orbifold fundamental group of the Coxeter orbifold $\pi_1(\hat{P})$ becomes the Coxeter group $\Gamma$. Moreover, the generators of $\Gamma$ are exactly the same as the projective reflections of the Coxeter polytope $P$.

\subsubsection{Vinberg's theory}
\label{ch:2.1.4}

Let $P$ be a $f$-sided Coxeter polytope whose sides are $S_i$ for $i=1,\cdots,f$. Recall that if the Coxeter polytope $P$ is given, then we can compute the convex polyhedral cone $K$ and each projective reflection $R_i$. Vinberg studied the properties of $K$ and $R_i$'s when the Coxeter polytope $P$ is given in \cite{Vin71}.

In this section, we will describe some propositions of Vinberg's works which will be used in the proof of the main theorem of this paper. More precisely, these propositions give us the conditions of nonemptiness of the $P$, in other words, (C5) of Vinberg's conditions.

\begin{proposition}[Proposition 13, \cite{Vin71}]\label{Vin71 Prop13}
Let $K$ be the convex polyhedral cone, and $M$ be the Cartan matrix. $L_\alpha$ be the space of linear relations among the $\alpha_i$, and $\L_r^+(M)$ be the convex cone in $L_r(M)$ consisting of the relations with non-negative coefficients, where $L_r(M)$ be the space of linear relations among rows of $M$. 
Assume that $M$ has positive diagonal entries. Then, $K$ is a $f$-sided cone if and only if
\[
L_\alpha \cap L_r^+(M) = \{0\}
\]
\end{proposition}

\begin{proposition}[Proposition 14, \cite{Vin71}]\label{Vin71 Prop14}

Let $P$ be a Coxeter polytope. Let $S_i$, $S_j$ be non-adjacent sides of $P$, and $\mathcal{S}$ be a subset of the index set of $P$ consisting of the indices of adjacent sides to $S_j$. Then, there exist numbers $d_j > 0$ and $d_s \geq 0$ ($s\in \mathcal{S}$) such that
\[
\alpha_i = -d_j \alpha_j + \sum_{s\in\mathcal{S}} d_s \alpha_s
\]

\end{proposition} 

We will use Proposition \ref{Vin71 Prop13} and Proposition \ref{Vin71 Prop14} in the proof of the general position case and the concurrent case, respectively, to find the space generating the Coxeter polytope with nonempty interior. Proposition \ref{Vin71 Prop13} will be used in the chapter \ref{ch:3.2} to show that any $\alpha_4$ linearly independent on the other $\alpha_j$'s induce the Coxeter polytope with nonempty interior. Proposition \ref{Vin71 Prop14} will be used in the chapter \ref{ch:3.1} and chapter \ref{ch:3.2} to find the condition when $\alpha_4$ does induce the nonempty interior of the Coxeter polytope. Furthermore, we will use the following proposition to compute the precise solution space in chapter \ref{ch:3.2}.

\begin{proposition}[Proposition 18, \cite{Vin71}]\label{Vin71 Prop18}

Let $\Gamma$ be a Coxeter group, $K$ be the convex polyhedral cone, and $M$ be the Cartan matrix. Then the followings are equivalent.

\begin{itemize}
    \item $K$ is strictly convex, that is, there is no hyperplane of $V$ containing two sides of $K$.
    \item $\operatorname{span}\{\alpha_i\}=V^*$.
\end{itemize}

$\Gamma$ satisfying one of these conditions is called \textit{reduced}. If $\Gamma$ is reduced, then
\[
\dim(\operatorname{span}\{v_i\}) = \operatorname{rank}(M).
\]
    
\end{proposition}

\subsection{Deformation spaces and character varieties}
\label{ch:2.2}

\subsubsection{Deformation spaces}
\label{ch:2.2.1}

Let $P$ be a $m$-dimensional polytope. A \textit{$k$-face} of $P$ is defined inductively. The $m$-face of $P$ is $P$ itself. We define $(m-1)$-face which is a side of $m$-face. Note that a $(m-1)$-face is a side of $P$, a $(m-2)$-face is a ridge of $P$, a 1-face of $P$ is an edge of $P$, and a 0-face of $P$ is a vertex of $P$. A \textit{flag of faces} of $P$ is a sequence $(F_0, F_1, \cdots, F_m)$ of faces of $P$ such that $F_k$ is a $k$-face of $P$ and $F_k$ is a side of $F_{k+1}$ for each $k$. We say that two polytopes $P$ and $Q$ are \textit{combinatorially equivalent polytopes} if there is a bijection between $P$ and $Q$ which preserves the flags of faces. Then, a \textit{combinatorial polytope} is a combinatorially equivalence class.

Let $P$ be a $n$-dimensional Coxeter polytope and $\hat{P}$ be the Coxeter orbifold. Since we want to focus on the deformation space of $\hat{P}$, not of $P$, we need the combinatorial equivalence of Coxeter orbifolds. We say that two Coxeter orbifolds $\hat{P}$ and $\hat{Q}$ are \textit{combinatorially equivalent orbifolds} if there is a bijection between $P$ and $Q$ that preserves the subsequence of flags of faces except for vertices and edges which are removed in the induced Coxeter orbifold. Let the \textit{projective isomorphism of Coxeter polytopes} be a projective automorphism of $\Sphere^n$ preserving the flags of faces. 

\begin{definition}\label{def:deformation space}

Let $\widetilde{\Def}(\hat{P})$ be the space consisting of isomorphism classes of Coxeter polytopes whose induced Coxeter orbifolds are combinatorially equivalent to $\hat{P}$. The \textit{deformation space (of convex real projective structures) of Coxeter orbifold $\hat{P}$}, denoted by $\Def(\hat{P})$, is the quotient space
\[
\Def(\hat{P}) := \widetilde{\Def}(\hat{P}) / \SL_\pm(n+1,\R)
\]
, that is, the space of projective isomorphism classes of Coxeter polytopes whose induced orbifolds are combinatorially equivalent to the Coxeter orbifold $\hat{P}$ if we consider $\SL_\pm(n+1,\R)$ as the automorphism group of convex real projective structures.

\end{definition}

We can topologize $\Def(\hat{P})$ by identifying as a subset of the quotient space
\[
((\R^{n+1})^*)^f \times (\R^{n+1})^f / \SL_\pm(n+1,\R) \times \R_+^f
\]
where $f$ is the number of sides of $P$. The action of the group is
\[
A \cdot (R_1, \cdots, R_f) = (AR_1A^{-1}, \cdots, AR_fA^{-1})
\]
and
\[
c_i \cdot (\alpha_i,v_i) = (c_i^{-1}\alpha_i, c_iv_i)
\]
for $A \in \SL_\pm(n+1,\R)$ and $c_i \in \R_+$.

\subsubsection{Character varieties}
\label{ch:2.2.2}

Let $\hat{P}$ be a Coxeter orbifold, $\Gamma=\pi_1(\hat{P})$, and $G=\SL_\pm(n,\R)$. Denote $\Hom(\Gamma, G)$ be the set of group homomorphisms from $\Gamma$ to $G$. Note that $G$ acts on $\Hom(\Gamma, G)$ by conjugation. Then, $\Hom(\Gamma, G)/G$ is the space of orbits of $G$ in $\Hom(\Gamma, G)$. Now, consider the space of closed orbits of $G$ in $\Hom(\Gamma, G)$.

\begin{definition} \label{def:character variety}

The \textit{character variety} is the space of equivalence classes
\[
\chi(\hat{P},G)=\{[\rho] \; | \; \rho \in \Hom(\Gamma,G)\}
\]
,where the equivalence relation is given if and only if their orbit closures intersect.
\[
\rho_1 \sim \rho_2 \Leftrightarrow G_{\rho_1} \cap G_{\rho_2}\neq \emptyset
\]

The construction of the character variety is called the \textit{real geometric invariant theory quotient}, simply \textit{$\R$-GIT quotient}, denoted by $\Hom(\Gamma, G) \sslash G$.
    
\end{definition}

A Lie group $G$ is \textit{reductive} if its unipotent radical is trivial, and \textit{linearly reductive} if every rational representation is semisimple.

\begin{lemma}[Lemma 2.1, \cite{CLM18} (original in \cite{Ric88})]\label{CLM18 Lemma2.1}

Let $G$ be a linearly reductive group (which includes $\SL_\pm(n+1,\R)$), and $\rho \in \Hom(\Gamma,G)$ be a representation. Then, the orbit $G \cdot \rho$ is closed if and only if the Zariski closure of $\rho(\Gamma)$ is a reductive subgroup of representation. Moreover,
\[\chi(\hat{P},G) = \{[\rho]\in \Hom(\Gamma,G)/G \; | \; \rho \text{ is reductive}\}.\]
    
\end{lemma}

\begin{theorem}[Theorem 2.1, \cite{Por23} (original in \cite{Lun75})]\label{Por23 Thm2.1}
    Let $G$ be a Lie group acting on $\Hom(\Gamma, G)$. Then the closure of each orbit by conjugation of $G$ has exactly one closed orbit. Moreover, $\chi(\hat{P}, G)$ is homeomorphic to a real semi-algebraic set.
\end{theorem}

Note that $\Hom(\Gamma, G)/G$ is not a Hausdorff space in general. By this theorem, it is remarkable that the $\R$-GIT quotient is a natural Hausdorffication from $\Hom(\Gamma, G)/G$.
\[
\Hom(\Gamma,G)/G \rightarrow \Hom(\Gamma, G) \sslash G = \chi(\hat{P}, G)
\]

Denote $\Hom^{\operatorname{ref}}(\Gamma, G)$ be the set of all representations of the Coxeter group as a reflection group in $V=\R^n$, and $\chi^{\operatorname{ref}}(\hat{P},G)$ be the subset of the character variety consisting of generated from the reflection group.

\subsubsection{The relations between deformation spaces and character varieties}
\label{ch:2.2.3}

Let $\hat{P}$ be a Coxeter orbifold and $G=\SL_\pm(n+1,\R)$. It is an interesting problem that the relation between the deformation space and the group homomorphism $\Hom(\pi_1(\hat{P}),G)$ or the character variety $\chi(\hat{P},G)$. $\widetilde{\Def}(\hat{P})$ is the space of projective polytopes with an obvious topology using the topology of projective polytopes with some edges and vertices removed combinatorially equivalent to $\hat{P}$ and the reflections.

Furthermore, we can focus on the subset of the deformation space consisting of semisimple representations. Recall the definition of semisimple representation. Let $V$ be a vector space and $\rho \in \Hom(\Gamma, G)$ be a representation. A subspace $W$ of $V$ is called \textit{invariant} if $\rho(A)w \in W$ for any $w \in W$ and $A \in G$. A representation with no nontrivial invariant subspaces is called \textit{irreducible} (or \textit{simple}). Then, a representation $\rho$ is \textit{semisimple} (or \textit{completely reducible}, or \textit{totally reducible}) if it is isomorphic to a direct sum of a finite number of irreducible representations. In addition, a representation $\rho \in \Hom(\pi_1(\hat{P}),G)$ is \textit{stable} if $\rho$ is semisimple and the centralizer of $\rho$ is finite. To analyze the deformation space $\Def(\hat{P})$, we will define some specific subspace.

\begin{definition} \label{def:Defss}

Let $\hat{P}$ be a Coxeter orbifold. Then, $\widetilde{\Defss}(\hat{P})$ is the subspace of $\widetilde{\Def}(\hat{P})$ of elements corresponding to semisimple representations, and $\Defss(\hat{P})$ is the quotient space of $\widetilde{\Defss}(\hat{P})$ by the same quotient topology.
    
\end{definition}

Similarly, we can define $\Def^{\operatorname{st}}(\hat{P})$ consisting of stable representations. If the given orbifold is compact, Ehresmann and Thurston showed that the given maps are local homeomorphisms. (Refer \cite{CLM18}, Theorem 3.2.)

\begin{theorem}[Ehresmann-Thurston principle]\label{Ehresmann-Thurston principle}

Suppose that $\hat{P}$ is compact. Then the following maps are local homeomorphisms.
\[
\widetilde{\Def}(\orb) \rightarrow \Hom(\pi_1(\hat{P}),G) \quad \text{and} \quad \Def^{\operatorname{st}}(\hat{P}) \rightarrow \Hom^{\operatorname{st}}(\pi_1(\hat{P}),G)/G
\]
where $\Hom^{\operatorname{st}}(\pi_1(\hat{P}),G)$ is the subset of $\Hom(\pi_1(\hat{P}),G)$ consisting of stable representations. 
\end{theorem}

For non-compact case, it is well-known that if $\hat{P}$ is a Coxeter orbifold only vertices are removed but no edges are removed, then the deformation space can be related with the character variety. However, $D^2(;n_1,n_2,n_3,n_4)\times \R$ is not only non-compact but also induced by removing edges. Therefore, we need the following proposition.

\begin{proposition}\label{Main Proposition}
    Let $\hat{P} = D^2(;n_1,n_2,n_3,n_4) \times \R$ and $G=\SL_\pm(4,\R)$. Then 
    \begin{itemize}
        \item $\widetilde{\Def}(\hat{P}) \rightarrow \Hom(\pi_1(\hat{P}),G)$
        \item $\Def(\hat{P}) \rightarrow \Hom(\pi_1(\hat{P}),G)/G$
        \item $\Defss(\hat{P}) \rightarrow \chi(\hat{P},G)=\Hom(\pi_1(\hat{P}),G) \sslash G$
    \end{itemize}
    are injective to an open subset of each codomain.
\end{proposition}

\begin{proof}
\[
\begin{tikzcd}
\widetilde{\Def} (\hat{P}) \arrow[rr] \arrow[d, "p'"] \arrow[dd, "r'"', bend right=49] &  & {\Hom(\pi(\hat{P}),G)} \arrow[d, "p"] \arrow[dd, "r", bend left=72]   \\
\Def(\hat{P}) \arrow[rr] \arrow[d, "q'"]              &  & {\Hom(\pi(\hat{P}),G)/G} \arrow[d, "q"] \\
\Defss(\hat{P}) \arrow[r]                        & {\chi(\hat{P},G)} \arrow[r, equal] & {\Hom(\pi_1(\hat{P}),G) \sslash G}                       
\end{tikzcd}
\]

We want to check the openness of those maps. Note that $\pi_1(\hat{P})$ is the Coxeter group. By the definition of the $\Hom^{\operatorname{ref}}(\pi_1(\hat{P}), G)$, there is $\rho \in \Hom^{\operatorname{ref}}(\pi_1(\hat{P}), G)$ which is generated by the projective reflections of the Coxeter polytope. Take a neighborhood of $\rho$ in $\Hom(\pi_1(\hat{P}), G)$. Note that $\rho$ can be one-to-one corresponded to the tuple $(\alpha,v)$ in $(V^* \times V)^4$ where $V=\R^4$. Then, we can choose another representation $\rho' \in \Hom^{\operatorname{ref}}(\pi_1(\hat{P}), G)$ by changing $(\alpha,v)$ with a sufficiently small difference. This implies that the three maps of the proposition are open maps.

Let $p : \Hom(\pi_1(\hat{P}),G) \rightarrow \Hom(\pi_1(\hat{P}),G)/G$ be a natural quotient map by $G$, and let $q$ be a quotient map such that the composition $r=q\circ p$ becomes the $\R$-GIT quotient map of $\Hom(\pi_1(\hat{P}),G)$. Let $p'$ be the quotient map defined in the definition of $\widetilde{\Def}(\hat{P})$. Since $p$ and $p'$ are the quotient maps generated by the same group $G$, $p$ and $p'$ commute with the first and the second maps of the proposition. In other words, the upper part of the diagram commutes. Similarly, $r$ and $r'$ commute with the first and the third maps of the proposition because we defined $\Defss(\hat{P})$ to hold this property.

Now, we want to check the injectivity of those maps. For the first map, note that the difference of some side of the convex polyhedral cone $K$ induces the different $\alpha_i$'s for some $i\in I$. In addition, the difference of some reflection in sides of the convex polyhedral cone $\gamma (K)$ for some $\gamma \in \Gamma = \pi_1(\hat{P})$ induces the different $v_i$'s for some $i\in I$. This implies that the first map is injective. For the second map, we can do a similar proof replacing the convex polyhedral cone with the combinatorial polytope. This implies that the second map is injective. For the third map, refer to the following theorem.

\begin{theorem}[Theorem 3.6, \cite{Ric88}]\label{Ric88 Thm3.6}

Let $G$ be a linearly reductive group, and $x\in G$. Then, $x$ is semisimple if and only if the orbit $G \cdot x$ is closed.
    
\end{theorem}

Note that the character $[\rho] \in \chi(\hat{P},G)$ is a class of a closed orbit of $G$. By Lemma \ref{CLM18 Lemma2.1}, the representation $\rho \in \Hom(\Gamma,G)$ that induces the character $[\rho]$ is irreducible, which is equivalent to $\rho$ being semisimple by Theorem \ref{Ric88 Thm3.6}. Therefore, we can conclude that there is a unique representation $\rho$ which is semisimple for each character $[\rho]$. This proves the injectivity of the third map.
\end{proof}

In Chapter \ref{ch:3}, we need to determine whether the given representation is semisimple or not. To address this, we will use the following lemma.

\begin{lemma}[Lemma 3.26, \cite{DGKLM23}]\label{DGKLM Lemma3.26}

A representation $\rho \in \Hom(\pi_1(\hat{P}),G)$ is semisimple if and only if $V = V_\alpha \oplus V_v$, where $V_\alpha$ is the intersection of the kernels of the $\alpha_j$'s, and $V_v = \operatorname{span}{v_i}$.

\end{lemma}

\subsection{Previous works}
\label{ch:2.3}

\subsubsection{Deformation spaces of Coxeter 2-orbifolds}
\label{ch:2.3.1}

Goldman studied the deformation space of convex projective orientable surfaces with negative Euler characteristic \cite{Gol90}. By cutting the surface into pairs of pants and cutting the pair of pants into two open triangles, each triangle has a union of four triangles, including itself, which is the projective invariant. By gluing these triangles with geodesics, the deformation space of the surface is homeomorphic to $\R^{-8\chi(S)}$.

\begin{theorem}[Theorem 1, \cite{Gol90}]\label{Gol90 Thm1}

Let $S$ be a compact projective surface with $\chi(S)<0$. Then, $Def(S)$ is a Hausdorff real analytic manifold of dimension $-8\chi(S)$.
    
\end{theorem}

Furthermore, Choi and Goldman generalized this theorem to the deformation of 2-orbifolds \cite{CG05}.

\begin{theorem}[Theorem A, \cite{CG05}]\label{CG05 ThmA}

Let $\orb^2$ be a compact 2-orbifold in $\RP^2$-structure with $\chi(\orb^2)<0$ and $\partial \orb^2 = \emptyset$. Then, $Def(\orb^2)$ is homeomorphic to a cell of dimension
\[
-8\chi(|\orb^2|) + (6k - 2k_2) + (3l - l_2)
\]
where $k$ is the number of cone-points, $k_2$ is the number of cone-points of order two, $l$ is the number of corner-reflectors, and $l_2$ is the number of corner-reflectors of order two.
    
\end{theorem}

The deformation spaces of 2-orbifolds are classified by this theorem. We will derive a corollary from this theorem to be used in this paper, which focuses on the Coxeter 2-orbifold induced from a convex Coxeter polygon.

\begin{corollary}[Corollary of Theorem \ref{CG05 ThmA}]\label{Cor of CG05}

Let $P$ be a convex Coxeter polygon with $f$ sides. Assume that $P$ is not a triangle. Let $\hat{P}$ be an induced Coxeter 2-orbifold. Assume that every edge order is finite and at least 3. Then, $\Def(\hat{P}) \cong \R^{3f-8}$.
    
\end{corollary}

\begin{proof}

Note that $\hat{P}$ can be identical to $D^2(;m_1,m_2,\cdots,m_f)$, in other words, a disk with $f$ corner-reflectors. We can compute that $\chi(|\hat{P}|)=1$ since it is homeomorphic to a topological disk. Therefore, we can show that
\[
\chi(\hat{P})=1-\frac{1}{2}\sum_{j=1}^{f}\left( 1-\frac{1}{m_j} \right) \leq
1-\frac{1}{2}\sum_{j=1}^{f}\left( 1-\frac{1}{3} \right) = 1-\frac{f}{3} < 0
.\]
Since the boundary of $\hat{P}$ is empty, we can apply the above theorem. Then, $\Def(\hat{P})$ is homeomorphic to a cell of dimension $-8+3f$.
\end{proof}

Therefore, we can conclude that the deformation space of the Coxeter orbifold induced by a convex Coxeter quadrilateral is homeomorphic to $\R^4$. We will use this fact in chapter \ref{ch:4.1}.

\subsubsection{Deformation spaces of Coxeter $n$-simplexes}
\label{ch:2.3.2}

Kac-Vinberg and Koszul found that the deformation space of $n$-simplex is homeomorphic to $\R^{n(n-1)/2}$ in \cite{VK67}. Suhyoung Choi reinterpreted the proof of the theorem with J. R. Kim in \cite{Cho06}. Here, we will reinterpret once again and use the method to prove Theorem \ref{Main Theorem}.

\begin{theorem}[Theorem 3, \cite{Cho06}]\label{Cho06 Thm3}
    Let $\triangle$ be a Coxeter $n$-simplex orbifold whose all edge orders are greater than or equal to 3. Let $G=\SL_\pm(n+1,\R)$. Then, $\Def(\triangle)$ is homeomorphic to $\R^{n(n-1)/2}$.
\end{theorem}

\begin{proof}

Refer that $\Hom_P(\pi_1(\triangle),G)/G$ is locally homeomorphic to the deformation space $\Def(\triangle)$ by \cite{Cho06}, Proposition 3. We will compute the deformation space $\Def(\triangle)$ using $\widetilde{\Def}(\triangle)$, which is homeomorphic to some open subset of $\Hom(\pi_1(\triangle),G)$, and this can be corresponded to some subset of $(V^* \times V)^{n+1}$. Here, $V = \R^{n+1}$. If two points are in the same orbit of $G$, then their Coxeter orbifolds are projectively equivalent to each other. This implies that they are at the same point in the deformation space. Therefore, we will compute the stabilizer of $G$, which is homeomorphic to the deformation space.

Denote a tuple 
\[(\alpha,v)=(\alpha_1, \cdots, \alpha_{n+1}, v_1, \cdots, v_{n+1}) \in (V^* \times V)^{n+1}\]
which is in the image of $\widetilde{\Def}(\triangle)$. Let $A \in G$ be an element of group action satisfying
\[\alpha_j \mapsto c_j^{-1} \alpha_j A^{-1} \quad \text{and} \quad v_j \mapsto c_j A v_j,\]
where $c_j \in \R_+$ are parameters to simplify the proof. (Here, $j = 1,\cdots,n+1$ is the index.)

The first claim is that any orbit of $G$ includes the point of $\alpha_j = e^*_j$ for every $j=1,\cdots,n+1$, where $e^*_j$ is the standard basis of the dual vector space $V^*$. It is sufficient to show that there is a tuple $(e^*_j, w)$ such that $c_j^{-1} \alpha_j A^{-1} = e^*_j$ and $c_j A v_j = w_j$. Then, $c_j^{-1} \alpha_j A^{-1} = e^*_j$ implies that $c_j^{-1} \alpha_j = A e^*_j$, hence we can represent $A$ as below.
\[
A=\begin{bmatrix}
    c_1^{-1} \alpha_1 \\ \vdots \\ c_{n+1}^{-1} \alpha_{n+1}
\end{bmatrix} \in \SL_\pm(n+1,\R)
\]
Since such a matrix $A$ exists for some $c_j$, the claim is true.

By the above assumption, the convex polyhedral cone $K$ becomes the first closed orthant, $\bigcap_{j=1}^{n+1}\{x_j \geq 0\}$. This is the fundamental domain of the reflection group $\Gamma$. Note that $\Sphere(K)$ is the Coxeter $n$-simplex.

Since the orbit of $G$ is not trivial in the subset of $\widetilde{\Def}(\triangle)$ with $\alpha_j = e^*_j$, we want to find stronger conditions. Note that
\[
\begin{bmatrix} w_1 & w_2 & \cdots & w_{n+1} \end{bmatrix} = \begin{bmatrix} c_1 A v_1 & c_2 A v_2 & \cdots & c_{n+1} A v_{n+1} \end{bmatrix} = D^{-1} M D
\]
\[ 
= \begin{bmatrix}
     2 & \frac{c_2}{c_1}M_{12} & \frac{c_3}{c_1}M_{13} & \cdots & \frac{c_n}{c_1}M_{1 \, n} & \frac{c_{n+1}}{c_1}M_{1 \, n+1} \\
    \frac{c_1}{c_2}M_{21} & 2 & \frac{c_3}{c_2}M_{23} & \cdots & \frac{c_n}{c_2}M_{2 \, n} & \frac{c_{n+1}}{c_2}M_{2 \, n+1} \\
    \frac{c_1}{c_3}M_{31} & \frac{c_2}{c_3}M_{32} & 2 & \cdots & \frac{c_n}{c_3}M_{3 \, n} & \frac{c_{n+1}}{c_3}M_{3 \, n+1} \\
    \vdots & \vdots & \vdots & \ddots & \vdots & \vdots \\
    \frac{c_1}{c_n}M_{n\,1} & \frac{c_2}{c_n}M_{n\,2} & \frac{c_3}{c_n}M_{n\,3} & \cdots & 2 & \frac{c_{n+1}}{c_n}M_{n\,n+1} \\
    \frac{c_1}{c_{n+1}}M_{n+1\,1} & \frac{c_2}{c_{n+1}}M_{n+1\,2} & \frac{c_3}{c_{n+1}}M_{n+1\,3} & \cdots & \frac{c_n}{c_{n+1}}M_{n+1\,n} & 2
\end{bmatrix}
\]
where $D=\diag(c_1,\cdots,c_{n+1})$ is a diagonal matrix and $M = [\alpha_i(v_j)]$ is the Cartan matrix. Assume that $w_1 = [2,-1,\cdots,-1]^{\operatorname{T}}$. Then, $c_2 = - c_1 M_{21}$, $\cdots$, and $c_{n+1} = -c_1 M_{n+1 \, 1}$. Then, $c_j$ ($j\neq1$) depends only on $c_1$. Consider the determinant of $A$. Then we can get the equation.
\[
\pm 1 = \det(A) = \frac{\det\begin{bmatrix} \alpha_1 & \cdots & \alpha_{n+1} \end{bmatrix}}{c_1 c_2 \cdots c_{n+1}}
 = \frac{\det\begin{bmatrix} \alpha_1 & \cdots & \alpha_{n+1} \end{bmatrix}}{c_1^{n+1}M_{21}\cdots M_{n+1 \, 1}}
\]
Since $M_{21}<0$, $\cdots$, $ M_{n+1 \, 1} < 0$ by Vinberg's theory and $\det\begin{bmatrix} \alpha_1 & \cdots & \alpha_{n+1} \end{bmatrix}$ is a constant number, such $c_1 >0$ is unique. This implies that $A$ is a constant matrix if we give conditions with $\alpha_j = e^*_j$ and $w_1 = [2,-1,\cdots,-1]^{\operatorname{T}} $. In other words, the subset of $(V^* \times V)^{n+1}$ with the above assumptions becomes a $G$-stabilizer. Therefore, the solution space of these conditions is homeomorphic to the deformation space.

Now, we will focus on the subset of $\Hom(\pi_1(\triangle),G)$ with $\alpha_j = e^*_j$ ($j=1,\cdots,n+1$) and $v_1 = [2,-1,\cdots,-1]^T$. More explicitly, we can denote a tuple $(\alpha,v)$ as below.
\[
[\alpha] = \begin{bmatrix} \alpha_1 \\ \vdots \\ \alpha_{n+1} \end{bmatrix} = I_{n+1} \quad \& \quad
[v] = \begin{bmatrix} v_1 & \cdots & v_{n+1} \end{bmatrix} =
\begin{bmatrix}
     2 & v_{12} & v_{13} & \cdots & v_{1 \, n} & v_{1 \, n+1} \\
    -1 & 2 & v_{23} & \cdots &v_{2 \, n} & v_{2 \, n+1} \\
    -1 & v_{32} & 2 & \cdots &v_{3 \, n} & v_{3 \, n+1} \\
    \vdots & \vdots & \vdots & \ddots & \vdots & \vdots \\
    -1 & v_{n+1 \, 2} & v_{n+1 \, 3} & \cdots & 2 & v_{n \, n+1} \\
    -1 & v_{n+1 \, 2} & v_{n+1 \, 3} & \cdots & v_{n+1 \, n} & 2
\end{bmatrix}
\]
Here, $I_{n+1}$ is the identity matrix. Therefore, $[v]$ is identified with the Cartan matrix $M$, in other words, $[v]$ satisfies Vinberg's conditions. This implies that entries of the first row $v_{12}$, $\cdots$, $v_{1 \, n+1}$ are constant by $(-1)v_{1 \, j} = 4\cos^2(\frac{\pi}{n_{1j}})=: \mu_{1j}$. Similarly, each entry of the upper diagonal depends only on the transpose entry satisfying $v_{ij}v_{ji}=\mu_{ij}$. Since we assumed that every order of the Coxeter $n$-simplex orbifold is greater than 2, every $v_{ij}$ is negative. These are all of conditions to find the solution space. The solution space is the product of $v_{ij}<0$ with $i<j$. Therefore, the deformation space is homeomorphic to $\R^{n(n-1)/2}$.
\end{proof}

Moreover, we can allow some edge orders equal to 2 by replacing the above equation by $M_{ij}=0$ and $M_{ji}=0$. This theorem directly induces that the deformation space of the Coxeter 3-simplex with finite orders is homeomorphic to $\R^3$.

\section[Proofs]{Proofs}
\label{ch:3}
\noindent

The motivation of the main theorem comes from the tetrahedron. We already saw that the deformation space of a tetrahedron with finite edge orders has dimension three. Consider a Coxeter 3-simplex with four sides, $S_1$, $S_2$, $S_3$, and $S_4$. Now, assume that there are two non-adjacent edges whose orders are infinite. Without loss of generality, assume that those two infinite edge orders are given on $n_{13}$ and $n_{24}$. The rest of the edge orders $n_{12}$, $n_{23}$, $n_{34}$, and $n_{14}$ are finite orders greater than or equal to 3. Then, the labeled polytope can be shown as the left side of Figure \ref{Figure general position case}. Since the Coxeter orbifold can be obtained by removing edges with infinite orders, the underlying space of the Coxeter orbifold can be shown as the right side of Figure \ref{Figure general position case}. This underlying space is topologically homeomorphic to the product of a disk and an open line. Furthermore, since the orbifold fundamental group of this orbifold is equal to the Coxeter group which is a reflection group with $n_{12}$, $n_{23}$, $n_{34}$, and $n_{14}$. Therefore, the Coxeter orbifold induced from the given labeled polytope can be presented by $D^2(;n_{12},n_{23}, n_{34}, n_{14})$. 

\begin{figure}[h]
    \centering
    \includegraphics[height=5cm]{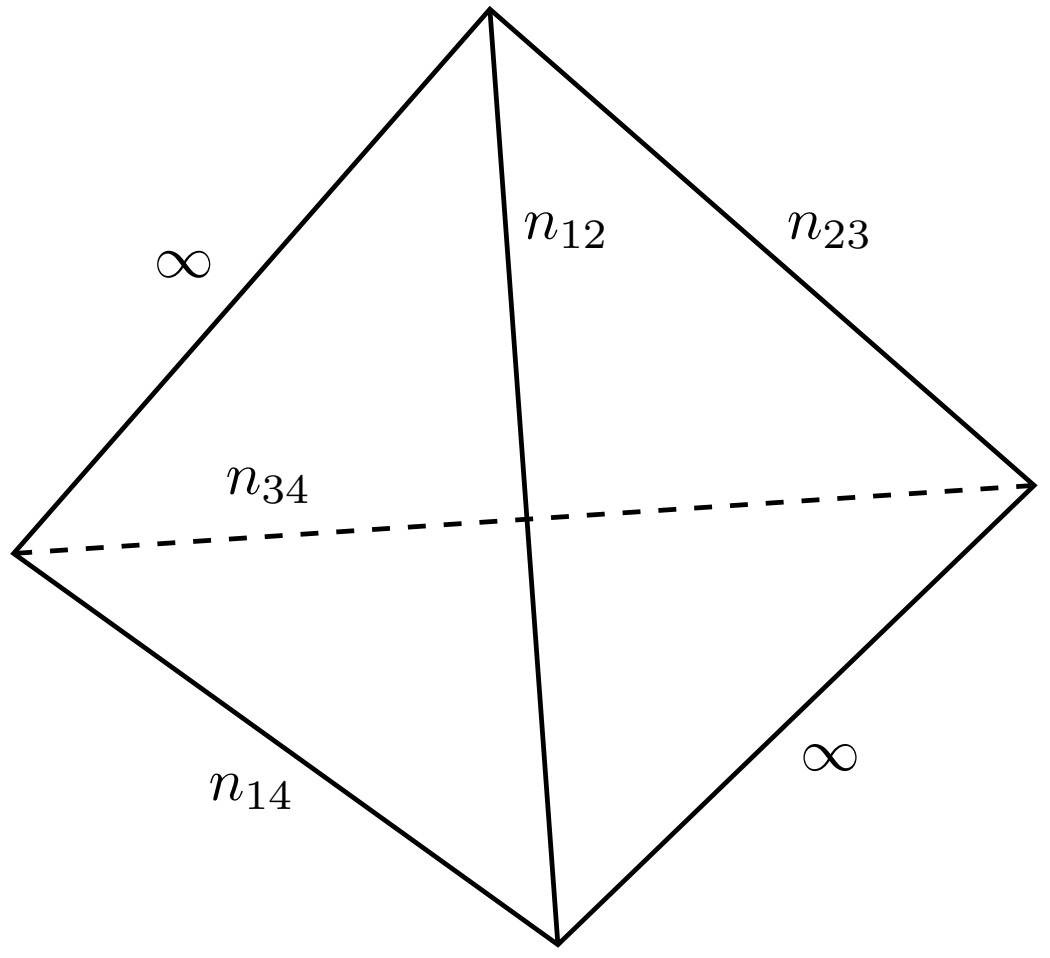}
    \includegraphics[height=5cm]{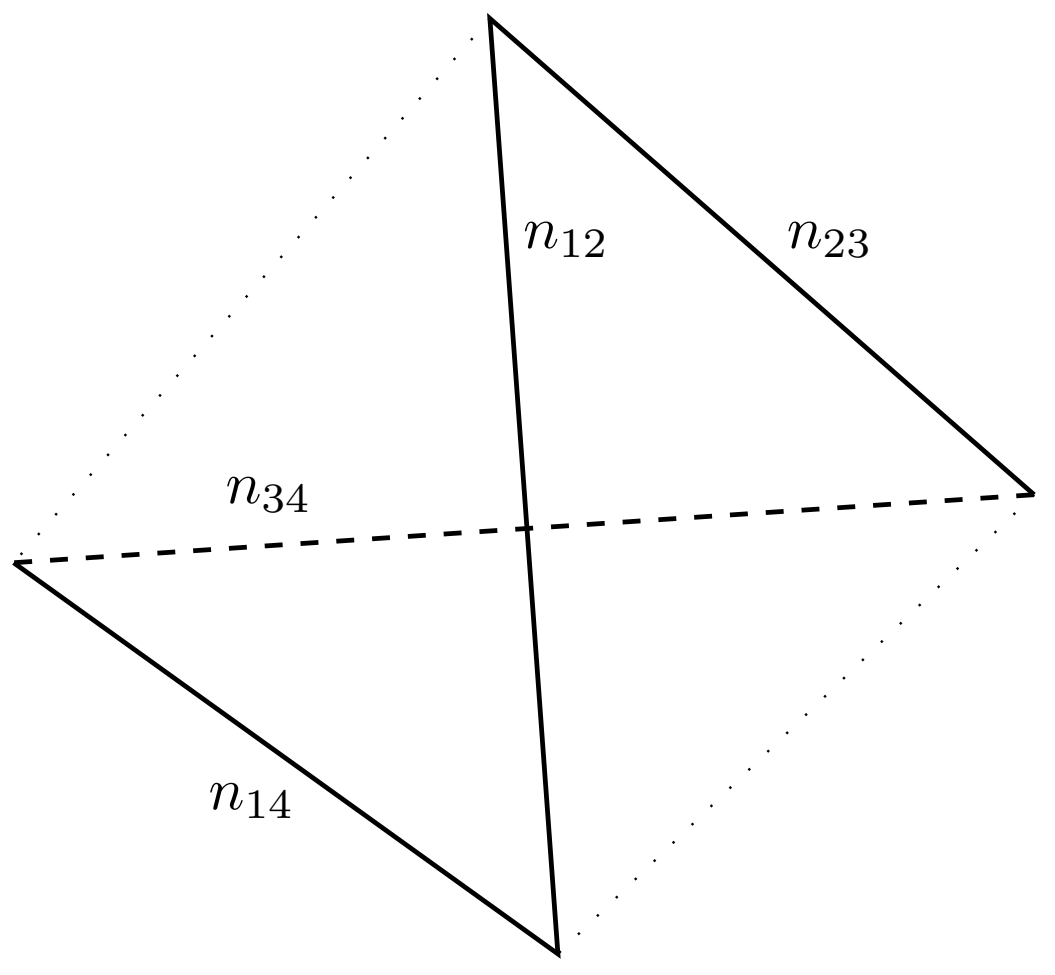}
    \caption{A labeled polytope and a Coxeter orbifold of the general position case}
    \label{Figure general position case}
\end{figure}

However, if we consider the projective reflections $R_i=\Id-\alpha_j\otimes v_j$ of each side $S_i$ from the above orbifold, $\alpha_1$, $\alpha_2$, $\alpha_3$, and $\alpha_4$ should be linearly independent in $V^*$. This is easily checked by the fact that the given Coxeter orbifold comes from the tetrahedron, that is, there is no pair of two sides which are concurrent in $\RP^3$. Because of this fact, we should check if there is a Coxeter orbifold $D^2(;n_{12},n_{23}, n_{34}, n_{14})$ induced from the Coxeter polytope whose linear functionals of reflections are linearly dependent in $V^*$. In fact, we can easily find an example.

Consider a quadrilateral in $\Sphere^3$. Suppose the Coxeter polytope $P$ such that each side of $P$ touches one side of the quadrilateral orthogonally. This Coxeter polytope $P$ has only two vertices $v_1$ and $v_2$, which are the preimage of double covering $q : \Sphere^3 \rightarrow \RP^3$ and $q(v_1)=q(v_2)$ which is the ideal point of $\RP^3$. In other words, $v_1$ is the antipodal point of $v_2$ in $\Sphere^3$. Then, the Coxeter orbifold can be induced by removing those two vertices, shown in the Figure \ref{Figure concurrent case}. For this orbifold, the linear functionals of the reflections are linearly dependent in $V^*$. To divide these two cases, we will call the former the \textit{general position case}, and the latter the \textit{concurrent case}.

\begin{definition} \label{def:general position case and concurrent case}

Let $P$ be a Coxeter polytope, and $R_j=\Id-\alpha_j \otimes v_j$ be its projective reflections. We say that $P$ is in \textit{general position case} if the linear functionals $\alpha_j$ are in general position in projective geometry, in other words, any $d+2$ of them do not lie on a $d$-dimensional projective plane. We say that $P$ is in the \textit{concurrent case} if $\alpha_j$ are linearly dependent in $V^*$.
    
\end{definition}

\begin{definition} \label{defgen and defcon}

Recall the notations of deformation spaces that we introduced in definition \ref{def:deformation space} and \ref{def:Defss}. Let $\Defgen(\hat{P})$ (resp. $\Defssgen(\hat{P})$) be the subspace of $\Def(\hat{P})$ (resp. $\Defss(\hat{P})$) corresponding to the general position case, and let $\Defcon(\hat{P})$ (resp. $\Defsscon(\hat{P})$) be the subspace of $\Def(\hat{P})$ (resp. $\Defss(\hat{P})$) corresponding to the concurrent case.
    
\end{definition}

\begin{figure}[h]
    \centering
    \includegraphics[height=5cm]{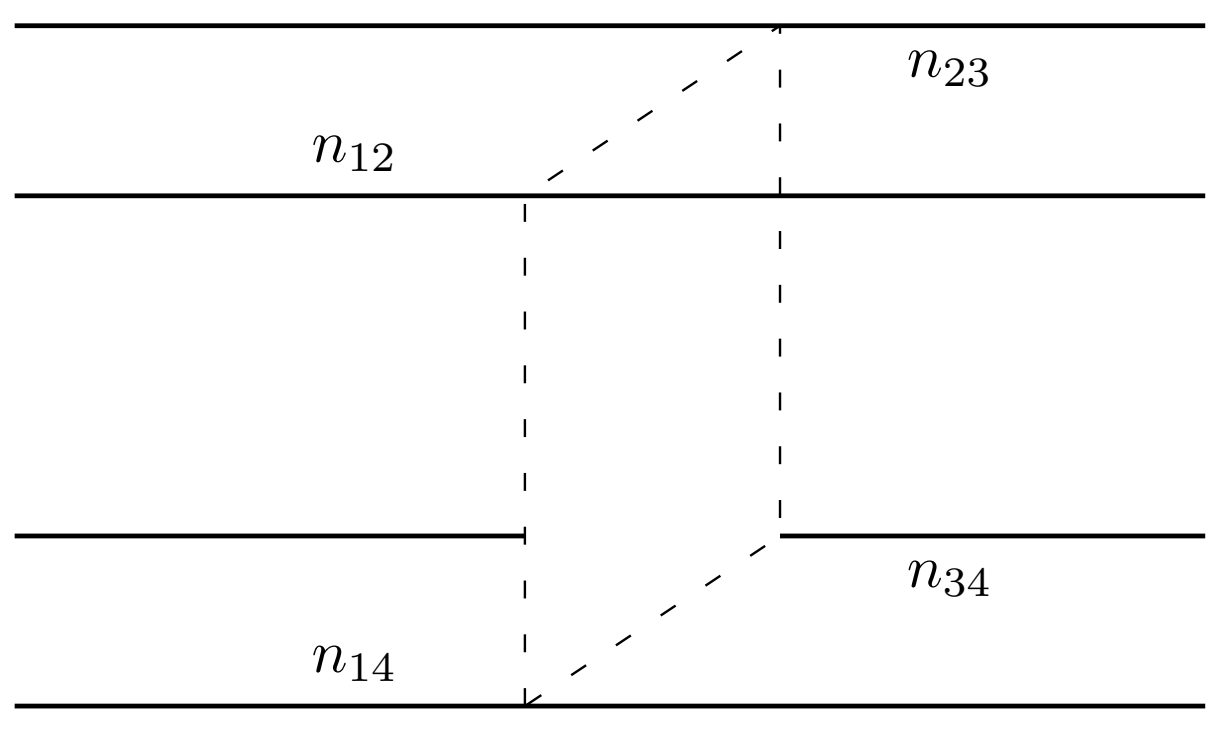}
    \caption{A Coxeter orbifold of the concurrent case}
    \label{Figure concurrent case}
\end{figure}

The proof follows a similar idea to the proof method of Theorem \ref{Cho06 Thm3}. By Proposition \ref{Main Proposition}, we can get relations between some subsets of the deformation space that have bijections to some open subset of the group homomorphism or its quotient spaces. Therefore, if we consider subsets of $\widetilde{\Def}(\hat{P})$, then it maps to some subsets of $\Hom(\pi_1(\hat{P}),G)$. We will find nicely parametrized subspaces, called $S \subset \widetilde{\Def}(\hat{P})$, so that their images can be related to $\Hom(\pi_1(\hat{P}),G)/G$ or $\chi (\hat{P},G)$ and each includes the image of $\Def(\hat{P})$ or $\Defss(\hat{P})$, respectively. For example, if the image of $S$ becomes a $G$-stabilizer, then $S$ is homeomorphic to $\Def(\hat{P})$. Even if not, we can also find a homeomorphism by computing the orbit of $G$ in the image of $S$.

We will prove the theorem in the following order. 
\begin{itemize}
    \item [1.] Find nice subsets of $\widetilde{\Defss}(\hat{P})$ (resp. $\widetilde{\Def}(\hat{P})$) that minimizes the dimension of the orbit of $G$ in $\Defss(\hat{P})$ (resp. $\Def(\hat{P})$) by imposing conditions on $\alpha$'s and $v$'s.
    \begin{itemize}
        \item[1-1.] To compute $\Defsscon(\hat{P})$, find the subset of $\widetilde{\Defss}(\hat{P})$ where $\alpha_j = e^*_j$, for $j=1,2,3$ and $\alpha_4 = e^*_1 - e^*_2 + e^*_3$, as described in chapter \ref{ch:3.1}.
    \item[1-2.] To compute $\Defss(\hat{P})$, find the subset of $\widetilde{\Defss}(\hat{P})$ using $S$ where $\alpha_j = e^*_j$ and $e^*_4(v_j)=0$ for $j=1,2,3$, and $e^*_1(v_4)=-1$, as described in chapter \ref{ch:3.2}.
    \item [1-3.] To compute $\Defgen(\hat{P})$, find the subset of $\widetilde{\Def}(\hat{P})$ where $\alpha_j=e^*_j$ for $j=1,2,3,4$, as described in chapter \ref{ch:3.3}.
    \end{itemize}
    
    \item [2.] Identify additional conditions on $\alpha$'s and $v$'s within the given subset and determine the corresponding solution space.
    \item [3.] Parametrize the given solution space.
\end{itemize}

\subsection{Proof of the main theorem (2)}
\label{ch:3.1}

Similar in chapter \ref{ch:2.3.2}, denote 
\[(\alpha,v) = (\alpha_1, \alpha_2, \alpha_3, \alpha_4, v_1, v_2, v_3, v_4) \in (V^* \times V)^4,\]
where $V=\R^4$. Let $A$ be an element of $G=\SL_\pm(4,\R)$ satisfying
\[ \alpha_j \mapsto c_j^{-1} \alpha_j A^{-1} \quad \text{and} \quad v_j \mapsto c_j A v_j, \]
where $c_j \in \R_+$ are parameters. Let 
\[(\beta,w) = (\beta_1, \beta_2, \beta_3, \beta_4, w_1, w_2, w_3, w_4)\]
be a tuple with group action $A \in G$ with
\[c_j^{-1}\alpha_{j}A^{-1}=\beta_j \quad \text{and} \quad c_j A v_j = w_j.\]
First, assume that $\beta_j = e_j^*$, $\forall j=1,2,3$. Then, $j$-th row of $A$ is identical to $c_j^{-1} \alpha_j$, except for $j=4$. Since $A \in \SL_\pm(4,\R)$, 4-th row is independent from the other rows, denote the 4-th row of $A$ be $c_4^{-1} \alpha_0$. Then, $\alpha_0$ is linearly independent with $\alpha_1$, $\alpha_2$, and $\alpha_3$ in $V^*$.

\[
A=\begin{bmatrix}
    c_1^{-1} \alpha_1 \\ c_2^{-1} \alpha_2 \\ c_3^{-1} \alpha_3 \\ c_{4}^{-1} \alpha_{0}
\end{bmatrix} \in \SL_\pm(4,\R)
\]
Then, compute $w_1, w_2, w_3$ and $w_4$.

\[
\begin{bmatrix} w_1 & w_2 & w_3 & w_4 \end{bmatrix} = \begin{bmatrix} c_1 A v_1 & c_2 A v_2 & c_3 A v_3 & c_4 A v_4 \end{bmatrix}
= D^{-1} \begin{bmatrix}
     2 & M_{12} & M_{13} & M_{14} \\
    M_{21} & 2 & M_{23} & M_{24} \\
    M_{31} & M_{32} & 2 & M_{34} \\
    \alpha_0(v_1) & \alpha_0(v_2) & \alpha_0(v_3) & \alpha_0(v_4)
\end{bmatrix} D
\]
It is remarkable that the 4-th row of $A$ depends on $\alpha_0$ and $c_4$. Here is the claim that we can restrict a subset with (4,1), (4,2), and (4,3) entries of the above matrix vanish. First, we want to show that $v_1$, $v_2$, and $v_3$ must be linearly independent in $V$. Since we consider the subset with $\alpha_j = e^*_j$ for $j=1,2,3$, the submatrix of the Cartan matrix $M_{3 \times 3}$ can be computed.
\[
M_{3 \times 3} = \begin{bmatrix} 2 & M_{12} & M_{13} \\ M_{21} & 2 & M_{23} \\ M_{31} & M_{32} & 2 \end{bmatrix}
\]
Then, consider the determinant of this matrix.
\[
\det(M_{3\times3}) = 2^3 + M_{21} M_{32} M_{13} + M_{31}M_{12}M_{23} - 2(M_{12}M_{21} + M_{23}M_{32} + M_{31}M_{13})
\]
Note that $M_{ij}$ are negative and $M_{12}M_{21} \geq 4$, hence $\det(M_{3\times3})<0$. Since $v_1$, $v_2$, and $v_3$ are linearly independent in $V$, then $\dim(\Ann\{v_1,v_2,v_3)\} = 1$. Choose $L \in \Ann\{v_1,v_2,v_3\} $ so that $\alpha_0 = t L$ for any $t\in \R-\{0\}$. In other words, we may assume that $\alpha_0(v_1)=\alpha_0(v_2)=\alpha_0(v_3)=0$.
\[
\begin{bmatrix} w_1 & w_2 & w_3 & w_4 \end{bmatrix}
= D^{-1} \begin{bmatrix}
     2 & M_{12} & M_{13} & M_{14} \\
    M_{21} & 2 & M_{23} & M_{24} \\
    M_{31} & M_{32} & 2 & M_{34} \\
    0 & 0 & 0 & t L(v_4)
\end{bmatrix} D
\]

For the concurrent case, note that sides of label 2 and 4 are not adjacent. By Proposition \ref{Vin71 Prop14}, $\alpha_4$ is determined by the following equation.
\[
\alpha_4 = -d_2 \alpha_2 + d_1 \alpha_1 + d_3 \alpha_3 \quad \text{ with } \quad d_2>0 \text{ and } d_1, d_3 \geq 0
\]
Moreover, we can induce that $d_1$ and $d_3$ are nonzero. Take a point $x$ in the edge between the side of label 1 and 2. Then, $\alpha_1(x) = \alpha_2(x)=0$ and $\alpha_{3}(x)$ and $\alpha_4(x)$ are nonzero. This implies that $d_3$ is nonzero. Similarly, if we take a point in the edge between the sides of label 2 and 3, we get that $d_1$ is nonzero. In other words, $\alpha_4$ is determined by the above equation with positive constants $d_1$, $d_2$, and $d_3$.
\[
\alpha_4 = -d_2 \alpha_2 + d_1 \alpha_1 + d_3 \alpha_3 \quad \text{ with } \quad d_1, d_2, d_3>0
\]

Now, we want to restrict the section with $\alpha_4 = e^*_1 - e^*_2 + e^*_3$. Assume that $c_4^{-1} \alpha_4 A^{-1} = e^*_1 - e^*_2 + e^*_3$, and let $\alpha_4 = d_1 \alpha_1 - d_2 \alpha_2 + d_3 \alpha_3$. Then, we get the below equation.
\[
c_1^{-1} \alpha_1 - c_2^{-1} \alpha_2 + c_3^{-1} \alpha_3 = (e_1^* - e_2^* + e_3^*)A = c_4^{-1} \alpha_4 = \frac{d_1}{c_4} \alpha_1 - \frac{d_2}{c_4} \alpha_2 + \frac{d_3}{c_4} \alpha_3
\]
Note that $\alpha_1$, $\alpha_2$, and $\alpha_3$ are linearly independent in $V^*$ (by the same reason with the linear independence of $v_1$, $v_2$, and $v_3$).
\[
c_1^{-1}=\frac{d_1}{c_4}, \quad c_2^{-1} = \frac{d_2}{c_4}, \quad c_3^{-1} = \frac{d_3}{c_4} \qquad \Leftrightarrow \qquad c_1 = \frac{c_4}{d_1}, \quad c_2 = \frac{c_4}{d_2}, \quad c_3 = \frac{c_4}{d_3}
\]
Apply to the determinant of $A$.
\[
\pm 1 = \det(A) = \frac{t}{c_1 c_2 c_3 c_4} \det \begin{bmatrix} \alpha_1 & \alpha_2 & \alpha_3 & L \end{bmatrix} = t \, \frac{d_1 d_2 d_3}{c_4^4} \det \begin{bmatrix} \alpha_1 & \alpha_2 & \alpha_3 & L \end{bmatrix}
\]
Since $d_1$, $d_2$, $d_3$, and $c_4$ are positive, $c_4$ is uniquely determined by $t$. In addition, $t$ is uniquely
determined by $c_4$ up to sign.

Thus, the subset $S$ with $\alpha_j = e_j^*$ ($j=1,2,3$) and $\alpha_4 = e^*_1 - e^*_2 + e^*_3$ is what we wanted. Choose $L=e_4^*$. Then, the action of $G$ in this $S$ (which is represented by $A$) is $\diag(\frac{d_1}{c_4}, \frac{d_2}{c_4}, \frac{d_3}{c_4}, \pm \frac{c_4^4}{d_1 d_2 d_3})$ for any $c_4>0$. We will compute the solution space with these assumptions, and then quotient by $G$ action as the diagonal matrix.

Let $v_{ij}$ be the $i$-th entry of $v_j$. Now, we consider Vinberg's conditions. Denote $\mu_{ij}:=\alpha_i(v_j)\alpha_{j}(v_i)=4\cos^2(\frac{\pi}{n_{ij}})$ for finite edge order $n_{ij}$. First of all, it is easy to replace $v_{21}$ and $v_{32}$ by $\mu_{12}/v_{12}$ and $\mu_{23}/v_{23}$, respectively. Then, a tuple $(\alpha,v)$ is as below.
\[
[\alpha] = \begin{bmatrix} \alpha_1 \\ \alpha_2 \\ \alpha_3 \\ \alpha_4 \end{bmatrix} = 
\begin{bmatrix}
    1 & 0 & 0 & 0 \\ 0 & 1 & 0 & 0 \\ 0 & 0 & 1 & 0 \\ 1 & -1 & 1 & 0
\end{bmatrix} \quad \& \quad
[v] = \begin{bmatrix} v_1 & v_2 & v_3 & v_4 \end{bmatrix} =
\begin{bmatrix}
    2 & v_{12} & v_{13} & v_{14} \\
    \frac{\mu_{12}}{v_{12}} & 2 & v_{23} & v_{24} \\
    v_{31} & \frac{\mu_{23}}{v_{23}} & 2 & v_{34} \\
    0 & 0 & 0 & v_{44}
\end{bmatrix}
\]
The others of Vinberg's conditions can be described as below.
\begin{itemize}
    \item[(1)] $v_{13} v_{31} \geq 4$
    \item[(2)] $2-\frac{\mu_{12}}{v_{12}}+v_{31}=\frac{\mu_{14}}{v_{14}}$
    \item[(3)] $\displaystyle \left( v_{12}-2+\frac{\mu_{23}}{v_{23}} \right) v_{24} \geq 4$ 
    \item[(4)] $v_{13}-v_{23}+2=\frac{\mu_{34}}{v_{34}}$
    \item[(5)] $v_{14} - v_{24} + v_{34} = 2$
    \item[(6)] $v_{31}$, $v_{12}$, $v_{13}$, $v_{23}$, $v_{14}$, $v_{24}$, $v_{34}$ are negative.
\end{itemize}
Here is the claim that (1) and (3) are always true by the other conditions. The conditions (2) and (4) can be simplified as below.
\[
v_{31} = \frac{\mu_{14}}{v_{14}} + \frac{\mu_{12}}{v_{12}} - 2 \quad \text{and} \quad v_{13} = v_{23} + \frac{\mu_{34}}{v_{34}} - 2
\]
Since $v_{14}$, $v_{12}$, $v_{23}$, $v_{34}$ are negative, we can conclude that $v_{31}<-2$ and $v_{12}<-2$, hence (1) is always true. Similarly, apply (5) into (3) with $v_{24} = v_{14} + v_{34} -2$, then (3) can be simplified as below.
\[
\left( v_{12} -2 +\frac{\mu_{23}}{v_{23}} \right) \left( v_{14}+v_{34}-2 \right) \geq 4
\]
Since each factor is less than -2 by negativity of variables, the multiple is always greater than 4, hence (4) is always true. Then, $[v]$ can be presented as below.

\[
[v] = \begin{bmatrix} v_1 & v_2 & v_3 & v_4 \end{bmatrix} =
\begin{bmatrix}
    2 & v_{12} & v_{23}+\frac{\mu_{34}}{v_{34}}-2 & v_{14} \\
    \frac{\mu_{12}}{v_{12}} & 2 & v_{23} & v_{14}+v_{34}-2 \\
    \frac{\mu_{14}}{v_{14}}+\frac{\mu_{12}}{v_{12}}-2 & \frac{\mu_{23}}{v_{23}} & 2 & v_{34} \\
    0 & 0 & 0 & 2
\end{bmatrix}
\]

Then, here is the Cartan matrix.

\[
M=
\begin{bmatrix}
    2 & v_{12} & v_{23}+\frac{\mu_{34}}{v_{34}}-2 & v_{14} \\
    \frac{\mu_{12}}{v_{12}} & 2 & v_{23} & v_{14}+v_{34}-2 \\
    \frac{\mu_{14}}{v_{14}}+\frac{\mu_{12}}{v_{12}}-2 & \frac{\mu_{23}}{v_{23}} & 2 & v_{34} \\
    \frac{\mu_{14}}{v_{14}} & v_{12}+\frac{\mu_{23}}{v_{23}}-2 & \frac{\mu_{34}}{v_{34}} & 2
    \end{bmatrix}
\]

Now, check the $G$-orbit. The action of $G$ is $A=\diag(\frac{1}{c_4}, \frac{1}{c_4}, \frac{1}{c_4}, \pm c_4^3)$, where $c_4 \in \R_+$. It is remarkable that $G$-action does not affect $\alpha_j$ for $j=1,2,3,4$. For example, $\alpha_1$ maps to the below by $A\in G$.
\[
c_1^{-1} \alpha_1 A^{-1} = \diag(1,1,1,\pm \frac{1}{c_4^4}) \alpha_1
\]
Note that $G$-action affects only the fourth entry of $\alpha_1$. Since we assumed that $\alpha_1 = e_1^*$, the fourth entry is always zero. Similarly for $\alpha_j$ with $j=2,3,4$, $G$-action does not affect $\alpha_j$. Moreover, $v_j$ with $j=1,2,3$ are also fixed by $G$ by $v_j \mapsto c_j A v_j$. For $j=4$, $G$ acts on $v_4$ as below.
\[
c_4 A v_4 = \diag(1, 1, 1, \pm c_4^4)
\begin{bmatrix}
    v_{14} \\ v_{24} \\ v_{34} \\ v_{44}
\end{bmatrix}
=
\begin{bmatrix}
    v_{14} \\ v_{24} \\ v_{34} \\ \pm c_4^4 v_{44}
\end{bmatrix}
\]
Moreover, if we check reflections, then $R_1$, $R_2$, $R_3$ are fixed. $G$ acts on $R_4$ as below.
\[
R_4 = \Id - \alpha_4 \otimes v_4 =
\begin{bmatrix}
    1-v_{14} & v_{14} & -v_{14} & 0 \\
    -v_{24} & 1+v_{24} & -v_{24} & 0 \\
    -v_{34} & +v_{34} & 1-v_{34} & 0 \\
    \mp c_4^4 \, v_{44} & \pm c_4^4 \, v_{44} & \mp c_4^4 \, v_{44} & 1
\end{bmatrix}
\]
By Lemma \ref{DGKLM Lemma3.26}, since $V_\alpha$ is one dimensional vector subspace, the representation among such $R_4$ is semisimple if and only if $v_{44}=0$. Therefore, the $\Defsscon(\hat{P})$ is homeomorphic to
\[
\{v_{12}<0\} \times \{v_{23}<0\} \times \{v_{14}<0\} \times \{v_{34} <0\}
\]
which is homeomorphic to $\R^4$.

\subsection{Proof of the main theorem (1), (3), and (5)}
\label{ch:3.2}

The general position case and the concurrent case induce the same Coxeter group as a reflection group. This implies that the orbifolds of those two cases are combinatorially equivalent. We want to show that there are no other cases except the general position case and the concurrent case. More precisely, if we assume that the dimension of $\operatorname{span}\{\alpha_1, \alpha_2, \alpha_3, \alpha_4\}$ is lower than or equal to 2, then the convex polyhedral cone $K$ does not exist. It is sufficient to show for dimension 2. We may assume that $\alpha_1=e_1^*$ and $\alpha_2=e_2^*$. Then, there is a natural projection $\R^4 \rightarrow \R^2$ with $(a,b,0,0) \mapsto (a,b)$, say $\alpha_j^\star \in \R^2$. Then $\alpha_j^\star$'s should induce a one-dimensional polyhedron in $\RP^1$ with four sides. However, even if we allow defining one-dimensional polyhedrons, it should have only two sides, which are the vertices of a closed interval. Therefore, it is enough to consider the dimension of span of $\alpha_j$'s with three or four, which are the concurrent case and the general position case, respectively.

We will see a subset $S$ such that we can see both the general position case and the concurrent case. Let $\alpha_j = e_j^*$ for $j=1,2,3$. In a similar way with chapter \ref{ch:3.1}, we can choose $L \in \Ann\{ v_1, v_2, v_3 \}$ and use $\alpha_0 = t \, L$ for $t \in \R-\{0\}$. Then, we can set $A \in G$ as below.
\[
A=\begin{bmatrix}
    c_1^{-1} \alpha_1 \\ c_2^{-1} \alpha_2 \\ c_3^{-1} \alpha_3 \\ c_{4}^{-1} t L
\end{bmatrix} \in \SL_\pm(4,\R)
\]
We can compute [w] using $A$ and $D=\diag(c_1, c_2, c_3, c_4)$
\[
\begin{bmatrix} w_1 & w_2 & w_3 & w_4 \end{bmatrix} = \begin{bmatrix} c_1 A v_1 & c_2 A v_2 & c_3 A v_3 & c_4 A v_4 \end{bmatrix}
= D^{-1} \begin{bmatrix}
     2 & M_{12} & M_{13} & M_{14} \\
    M_{21} & 2 & M_{23} & M_{24} \\
    M_{31} & M_{32} & 2 & M_{34} \\
    0 & 0 & 0 & t L(v_4)
\end{bmatrix} D
\]
Assume that (2,1), (3,1), and (1,4) entries of the above matrix are fixed to $-1$. In other words, let $c_2 = -c_1 M_{21}$, $c_3 = -c_1 M_{31}$, and $c_4 = -c_1 / M_{14}$. If we consider the determinant of the matrix,
\[
\pm 1 = \det(A) = \frac{t}{c_1 c_2 c_3 c_4} \det
\begin{bmatrix} \alpha_1 & \alpha_2 & \alpha_3 & L \end{bmatrix}
= -t \frac{M_{14}}{c_1^4 M_{21} M_{31}} \det
\begin{bmatrix} \alpha_1 & \alpha_2 & \alpha_3 & L \end{bmatrix}.
\]
Hence, $t$ is uniquely determined by $c_1$ up to sign. Then, $[w]$ can be simplified as below.
\[
\begin{bmatrix} w_1 & w_2 & w_3 & w_4 \end{bmatrix}
= \begin{bmatrix}
     2 & -M_{12}M_{21} & -M_{13}M_{31} & -1 \\
    -1 & 2 & \frac{M_{31}}{M_{21}} M_{23} & \frac{M_{14}}{M_{21}} M_{24} \\
    -1 & \frac{M_{21}}{M_{31}} M_{32} & 2 & \frac{M_{14}}{M_{31}} M_{34} \\
    0 & 0 & 0 & t L(v_4)
\end{bmatrix}
\]

Therefore, we can conclude that the subset $S$ with assumptions $\alpha_j = e_j^*$ ($j=1,2,3$) and
\[
v_1 = \begin{bmatrix} 2 \\ -1 \\ -1 \\ 0 \end{bmatrix} \quad
v_2 = \begin{bmatrix} \cdot \\ 2 \\ \cdot \\ 0 \end{bmatrix} \quad 
v_3 = \begin{bmatrix} \cdot \\ \cdot \\ 2 \\ 0 \end{bmatrix} \quad 
v_4 = \begin{bmatrix} -1 \\ \cdot \\ \cdot \\ \cdot \end{bmatrix}
\]
is exactly what we wanted. We will call this section the \textit{standard position}. Then, we can denote a tuple $(\alpha,v)$ as below.
\[
[\alpha] = \begin{bmatrix} \alpha_1 \\ \alpha_2 \\ \alpha_3 \\ \alpha_4 \end{bmatrix} = 
\begin{bmatrix}
    1 & 0 & 0 & 0 \\ 0 & 1 & 0 & 0 \\ 0 & 0 & 1 & 0 \\ a_1 & a_2 & a_3 & a_4
\end{bmatrix} \quad \& \quad
[v] = \begin{bmatrix} v_1 & v_2 & v_3 & v_4 \end{bmatrix} =
\begin{bmatrix}
    2 & -\mu_{12} & v_{13} & -1 \\
    -1 & 2 & v_{23} & v_{24} \\
    -1 & \frac{\mu_{23}}{v_{23}} & 2 & v_{34} \\
    0 & 0 & 0 & v_{44}
\end{bmatrix}
\]
From the standard position, it is remarkable that the general position case and the concurrent case are determined by that the fourth entry of $\alpha_4$ and the fourth entry of $v_4$ (say $a_4$ and $v_{44}$, respectively) are zero or not. To classify the cases more precisely, we will say (\rom{1}) case for $a_4\neq 0$, (\rom{2}) case for $a_4=0$ and $v_{44}\neq0$, and (\rom{3}) case for $a_4=0$ and $v_{44}=0$. Then, the general position case is (\rom{1}), and the concurrent case is the union of (\rom{2}) and (\rom{3}).

Consider (\rom{1}) case. We want to find the conditions under which $\{\alpha_j\}$ forms a non-empty 4-sided convex polyhedral cone $K$. By Proposition \ref{Vin71 Prop13}, if we satisfy the below equation, then $\{\alpha_j\}$ forms such $K$.
\[
L_\alpha \cap L^+_r (M^\circ) = \{0\}
\]
Recall that $L_\alpha$ is the space of linear relations among the $\alpha_j$.
\[
L_\alpha := \{ (d_1, d_2, d_3, d_4) \in \R^4 \; | \; d_1 \alpha_1 + d_2 \alpha_2 + d_3 \alpha_3 + d_4 \alpha_4 = 0 \}
\]
Since $\alpha_1$, $\alpha_2$, $\alpha_3$, and $\alpha_4$ are linearly independent if $a_4 \neq 0$, $L_{\alpha}$ consists of only the origin. This implies that any $\alpha_4$ with $a_4 \neq 0$ forms the Coxeter orbifold of the general position.

Now, consider (\rom{2}) and (\rom{3}) cases. By Proposition \ref{Vin71 Prop14}, similar to the proof in chapter \ref{ch:3.1}, $\alpha_4$ satisfies the below equation
\[
\alpha_4 = -d_2 \alpha_2 + d_1 \alpha_1 + d_3 \alpha_3
\]
with $d_1$, $d_2$, $d_3$ being positive. Since we assumed that $\alpha_j = e_j^*$ for $j=1,2,3$, we get $a_1 >0$, $a_3 > 0$, and $a_2 < 0$.

Therefore, we compute the solution space as below.
\[
\begin{bmatrix}
    2 & -\mu_{12} & v_{13} & -1 \\
    -1 & 2 & v_{23} & v_{24} \\
    -1 & \frac{\mu_{23}}{v_{23}} & 2 & v_{34} \\
    2a_1-a_2-a_3 & -a_1\mu_{12}+2a_2+a_3\frac{\mu_{23}}{v_{23}} & a_1v_{13}+a_2 v_{23} + 2a_3 & -a_1 + a_2v_{24} + a_3v_{34} + a_4v_{44}
\end{bmatrix}
\]
\[
=\begin{bmatrix}
    2 & -\mu_{12} & v_{13} & -1 \\
    -1 & 2 & v_{23} & v_{24} \\
    -1 & \frac{\mu_{23}}{v_{23}} & 2 & v_{34} \\
    -\mu_{14} & -a_1\mu_{12}+2a_2+a_3\frac{\mu_{23}}{v_{23}} & \frac{\mu_{34}}{v_{34}} & 2
\end{bmatrix}
\]

\begin{itemize}
    \item[(1)] $-v_{13} \geq 4$
    \item[(2)] $(-1)(2a_1 - a_2 - a_3) = \mu_{14}$
    \item[(3)] $v_{24} \left(-a_1\mu_{12}+2a_2+a_3\frac{\mu_{23}}{v_{23}}\right) \geq 4$
    \item[(4)] $a_1 v_{13} + a_2 v_{23} + 2a_3 = \frac{\mu_{34}}{v_{34}}$
    \item[(5)] $-a_1 + a_2 v_{24} + a_3 v_{34} + a_4 v_{44} = 2$
    \item[(6)] $v_{23}$, $v_{24}$, $v_{34}$ are negative.
    \item[(7)] $a_4 \neq 0$ or ($a_4 =0$ and $a_1>0$, $a_2<0$, $a_3>0$)
\end{itemize}
The seventh condition serves as the criterion that distinguishes between the general position case and the concurrent case. Note that this subset $S$ has an orbit of $G$, similarly to the concurrent case. Choose $L$ with $e_4^*$ so that the $G$-action is represented $A=\diag(c_1^{-1}, c_2^{-1}, c_3^{-1}, t c_4^{-1})$, where $t = \pm c_1^4 \in \R-\{0\}$ is derived by the determinant. Then, for $j=1,2,3$, $\alpha_j$ and $v_j$ are fixed by $G$.
\[
\alpha_j \mapsto c_j^{-1} \alpha_j A^{-1} = \diag(1,1,1,t^{-1}) \alpha_j = \alpha_j \quad \& \quad
v_j \mapsto c_j A v_j = \diag(1,1,1,t) v_j = v_j
\]
For $j=4$, the action of $G$ can be parameterized by $c_1$.
\[
\alpha_4 = \begin{bmatrix} a_1 \\ a_2 \\ a_3 \\ a_4 \end{bmatrix} \mapsto c_4^{-1} \begin{bmatrix} a_1 \\ a_2 \\ a_3 \\ a_4 \end{bmatrix} \diag(c_1, c_2, c_3, t^{-1} c_4) =
\begin{bmatrix} a_1 \\ a_2 \\ a_3 \\ t^{-1} a_4
\end{bmatrix}
\]
\[
v_4 = \begin{bmatrix} v_{14} & v_{24} & v_{34} & v_{44} \end{bmatrix} \mapsto c_4 A v_4 =
\begin{bmatrix} v_{14} & v_{24} & v_{34} & t v_{44} \end{bmatrix}
\]
Note that $G$ acts only on the fourth entries of $\alpha_4$ and $v_4$ with the multiple of two being constant. 

Now, we check the reflections $R_j$. It is easy to check that the reflections $R_1$, $R_2$, and $R_3$ are fixed under $G$ action, represented by $t \in \R-\{0\}$ for every case. First, we will see $R_4$ of (\rom{1}) case, denote $R_4^{(\rom{1})}$.
\[
R_4^{(\rom{1})} = \Id - \alpha_4 \otimes v_4 =
\begin{bmatrix}
    1+a_1 & a_2 & a_3 & -t a_4 \\
    -a_1 v_{24} & 1-a_2 v_{24} & -a_3 v_{24} & -t a_4 v_{24} \\
    -a_2 v_{34} & -a_2 v_{34} & 1-a_3 v_{34} & -t a_4 v_{34} \\
    -\frac{a_1}{t}v_{44} & -\frac{a_2}{t}v_{44} & -\frac{a_3}{t}v_{44} & 1-a_4 v_{44}
\end{bmatrix}
\]
If we see the section with $a_4$ and $v_{44}$ as the coordinate plane, quotient by $G$ is equivalent to quotient by hyperbola of $a_4 v_{44} = c \in \R-\{0\}$ if $v_{44}\neq 0$. If $v_{44}=0$, then quotient by $G$ is equivalent to $a_4$-axis. Let the case denote ($\rom{1}'$).
\[
R_4^{(\rom{1}')} = \Id - \alpha_4 \otimes v_4 =
\begin{bmatrix}
    1+a_1 & a_2 & a_3 & -t a_4 \\
    -a_1 v_{24} & 1-a_2 v_{24} & -a_3 v_{24} & -t a_4 v_{24} \\
    -a_2 v_{34} & -a_2 v_{34} & 1-a_3 v_{34} & -t a_4 v_{34} \\
    0 & 0 & 0 & 1
\end{bmatrix}
\]
\begin{figure}[h]
    \centering
    \includegraphics[height=5cm]{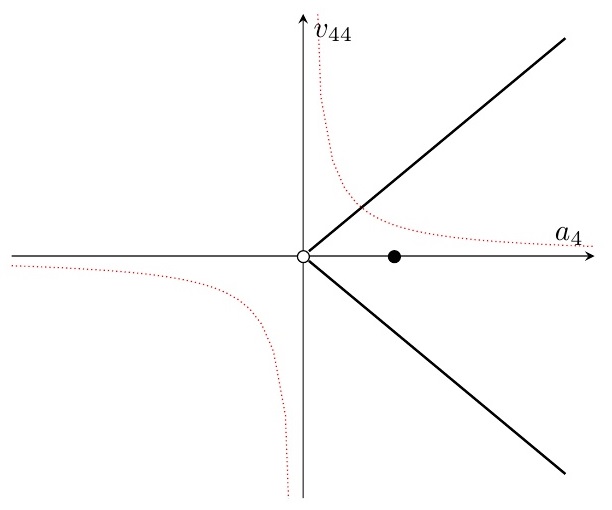}
    \caption{The section of $\Def(\hat{P})$ of the general position case with two coordinates $a_4$ and $v_{44}$}
    \label{Figure graph of general position case}
\end{figure}

The solution space of the general position case quotient by hyperbolas is the union of two lines and one point as the Figure \ref{Figure graph of general position case}. Moreover, if we find the subset of the solution space consisting of $a_4\neq0$ and $v_{44}=0$, by Proposition \ref{Vin71 Prop18}, then $\operatorname{rank}(M)=\dim(\operatorname{span}\{v_i\})=3$. Hence, the determinant of the Cartan matrix is zero. If we compute the determinant directly, then
\[
\det(M) = (4-M_{13}M_{31})(4-M_{24}M_{42})-E
\]
for some $E>0$. (Refer to the details in chapter \ref{ch:5.2}.) Then, the solution space of $(\rom{1})$ is homeomorphic to $\R^4$, same as the concurrent case.

Next, focus on $R_4$ of (\rom{2}) and (\rom{3}) case, denote $R_4^{(\rom{2})}$, $R_4^{(\rom{3})}$, respectively.
\[
R_4^{(\rom{2})} = \Id - \alpha_4 \otimes v_4 =
\begin{bmatrix}
    1+a_1 & a_2 & a_3 & 0 \\
    -a_1 v_{24} & 1-a_2 v_{24} & -a_3 v_{24} & 0 \\
    -a_2 v_{34} & -a_2 v_{34} & 1-a_3 v_{34} & 0 \\
    -\frac{a_1}{t}v_{44} & -\frac{a_2}{t}v_{44} & -\frac{a_3}{t}v_{44} & 1
\end{bmatrix}
\]
\[
R_4^{(\rom{3})} = \Id - \alpha_4 \otimes v_4 =
\begin{bmatrix}
    1+a_1 & a_2 & a_3 & 0 \\
    -a_1 v_{24} & 1-a_2 v_{24} & -a_3 v_{24} & 0 \\
    -a_2 v_{34} & -a_2 v_{34} & 1-a_3 v_{34} & 0 \\
    0 & 0 & 0 & 1
\end{bmatrix}
\]
By Lemma \ref{DGKLM Lemma3.26}, we can see that $R_4^{(\rom{1})}$ and $R_4^{(\rom{3})}$ are semisimple, but $R_4^{(\rom{1}')}$ and $R_4^{(\rom{2})}$ are not.

\begin{figure}[h]
    \centering
    \includegraphics[height=5cm]{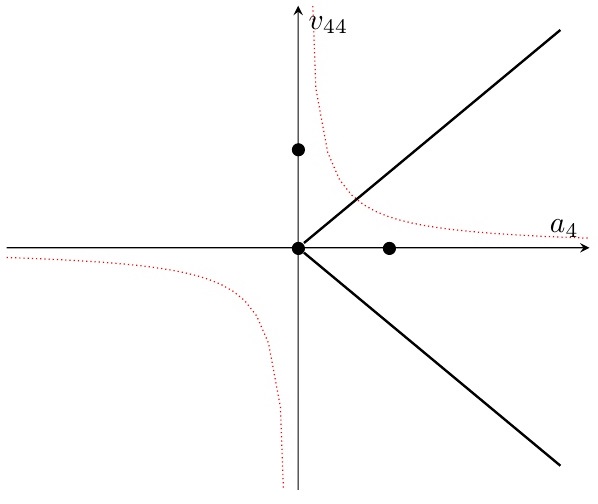}
    \includegraphics[height=5cm]{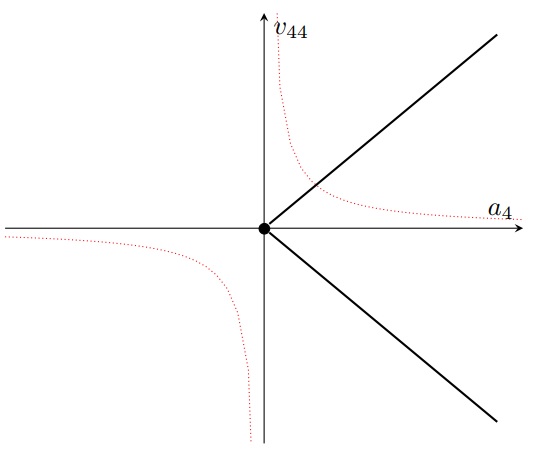}
    \caption{The section of $\Def(\hat{P})$ and $\Defss(\hat{P})$ with two coordinates $a_4$ and $v_{44}$}
    \label{Figure graph standard}
\end{figure}

By the Figure \ref{Figure graph standard}, we can notice that $\Defss(\hat{P})$ of the general position case becomes disconnected by the concurrent case, which is the origin of the coordinate system. Therefore, we can conclude that the deformation space of the general position case with semisimple representations consists of two connected components. In addition, two points representing the case of $(\rom{1}')$ or $(\rom{2})$ are non-Hausdorff components in $\Def(\hat{P})$.

\begin{remark}
The geometric meanings of the non-Hausdorff components are as follows. In case $(\rom{1}')$, the kernels of linear functionals form a tetrahedron, whereas the eigenvectors of the projective reflections span a codimension-one subspace. Therefore, the union of the images of the convex polyhedral cone under projective reflections  $\cup_{\gamma\in \Gamma}\gamma K$ has codimension one. Therefore, case $(\rom{1}')$ can be regarded as the concurrent case. In case $(\rom{2})$, the kernels of linear functionals form the concurrent case, but one eigenvector of projective reflections is tilted. Nevertheless, although one eigenvector of the projective reflections is tilted, the image of the convex polyhedral cone $\gamma K$ under each projective reflection is identical to its image in the concurrent case.
\end{remark}

Now, find the solution space of this case. Since we showed that the multiple $a_4 v_{44}$ is constant in the same $G$-orbit, we can consider $a_4 v_{44}$ as a real variable. Then, we can simplify Vinberg's conditions of this subset $S$ using $M_{13}M_{31}$ and $M_{24}M_{42}$. Here, $v_{13}=-M_{13}M_{31}$.
\[
\begin{bmatrix}
    a1 \\ a2 \\ a3 \\ a_4 v_{44}
\end{bmatrix}
=
\begin{bmatrix}
    2 & -1 & -1 & 0 \\
    -\mu_{12} & 2 & \frac{\mu_{23}}{v_{23}} & 0 \\
    v_{13} & v_{23} & 2 & 0 \\
    -1 & v_{24} & v_{34} & 1
\end{bmatrix}^{-1}
\begin{bmatrix}
    -\mu_{14} \\ \frac{M_{24}M_{42}}{v_{24}} \\ \frac{\mu_{34}}{v_{34}} \\ 2
\end{bmatrix}
\]
Then, $a_1$, $a_2$, $a_3$, and $v_{12}$ are determined by $v_{23}$, $v_{24}$, $v_{34}$, $T_{13} := M_{13}M_{31}$, and $T_{24} := M_{24}M_{42}$. The solution space is the below.
\[
\{v_{23}<0\} \times \{v_{24}<0\} \times \{v_{34}<0\} \times \{T_{13} \geq 4\} \times \{ T_{24} \geq 4 \}
\]
Therefore, $\Defss(\hat{P})$ is homeomorphic to $\R^3 \times [4,\infty)^2$.

\subsection{Proof of the main theorem (4)}
\label{ch:3.3}

While proving this theorem, there was an attempt to take the subset with $\alpha_j = e^*_j$ with $j=1,2,3,4$, and $v_1 = [2,-1,-1,-1]^{\operatorname{T}}$, same as the proof in chapter \ref{ch:2.3.2}. We expected that the solution space induced by this subset implies the deformation space of the general position case. However, this subset includes some non-semisimple representations. This makes some technical issues for the meaning of the general position case.

The basic idea of the proof is similar to what we have done, hence we omit the well-definedness of the subset $S$ in detail. Focus on the subset $S$ with $\alpha_j = e^*_j$ with $j=1,2,3,4$ and $v_1 = [2,-1,-1,-1]^{\operatorname{T}}$. Since $[\alpha]$ is an identity matrix, $[v]$ is identified with the Cartan matrix $M$. If we apply Vinberg's conditions, then we can compute $[v]$ directly.
\[
[v] = \begin{bmatrix} v_1 & v_2 & v_3 & v_4 \end{bmatrix} =
\begin{bmatrix}
     2 & -\mu_{12} & v_{13} & -\mu_{14} \\
    -1 & 2 & v_{23} & v_{24} \\
    -1 & \frac{\mu_{23}}{v_{23}} & 2 & v_{34} \\
    -1 & v_{42} & \frac{\mu_{34}}{v_{34}} & 2
\end{bmatrix}
\]

These conditions are equivalent to the belows.
\begin{itemize}
    \item[(1)] $(-1)v_{13} \geq 4$.
    \item[(2)] $v_{24}v_{42} \geq 4$.
    \item[(3)] $v_{42}$, $v_{23}$, $v_{24}$, $v_{34}$ are negative. 
\end{itemize}

Then, we can find that the solution space is
\[
\{v_{13} \leq -4\} \times \{v_{23}<0\} \times \{v_{34}<0\} \times \{ (v_{24},v_{42})\in \R^2 \; | \; v_{24}v_{42} \geq 4, \; v_{24}<0 \}
\]
Note that if we use $M_{13}M_{31}$ and $M_{24}M_{42}$, then $v_{13}=-M_{13}M_{31}=-T_{13}$ and $\displaystyle v_{42}=\frac{M_{24}M_{42}}{v_{24}}=\frac{T_{24}}{v_{24}}$. Hence, we can reparametrize the solution space
\[
\{T_{13}\geq4\} \times \{v_{23}<0\} \times \{v_{24}<0\} \times \{v_{34}<0\} \times \{T_{24}\geq4\} 
\]
, which is homeomorphic to $\R^3 \times \left[4,\infty\right)^2$.

It is remarkable that this solution space is homeomorphic to $\Defss(\hat{P})$. This result can be explained by the following proposition.

\begin{proposition}[Fact 3.28, \cite{DGKLM23}]\label{DGKLM Fact3.28}

For any $M \in \Mat(n,\R)$ of rank $\leq \dim(V)$ that is (weakly) compatible with the Coxeter group $\Gamma$, there exists a unique conjugacy class of semisimple representations $\rho \in \Hom(\Gamma, G)$ with Cartan matrix $M$.

\end{proposition}

We took the subset of $\widetilde{\Def}(\hat{P})$ so that its quotient space becomes the stabilizer of $G=\SL_\pm(4,\R)$. This implies that any two non-semisimple representations do not share a common character. Consequently, each point in $\Defgen(\hat{P})$ corresponding to a non-semisimple representation has a unique conjugacy class of semisimple representations, and there is a unique point in $\Defss(\hat{P})$ that belongs to the same conjugacy class of semisimple representations.

\section[Properties of the results]{Properties of the results}
\label{ch:4}

\subsection{Results in character varieties}
\label{ch:4.1}

Porti showed that the neighborhood of character varieties of $\pi_1(\orb^2)$ in $\SL_\pm(n+1,\R)$ can be computed from $\SL_\pm(n,\R)$ in \cite{Por23}. There is a theorem for orientable 2-orbifold (\cite{Por23}, Theorem 1.1), and another for non-orientable 2-orbifold. Note that the Coxeter orbifold with the reflection group is non-orientable since the corner-reflectors of the Coxeter orbifold do not preserve orientation. We will examine the theorem for the non-orientable case. 

Let $\orb^2$ be a compact, non-orientable 2-orbifold with negative Euler characteristic $\chi(\orb^2)<0$. A \textit{type preserving embedding} is an embedding from $\SL_\pm(n,\R)$ to $\SL_\pm(n+1,\R)$ satisfying the following condition.
\[
\SL_\pm(n,\R) \hookrightarrow \SL_\pm(n+1,\R) \quad ; \quad A \mapsto \begin{pmatrix} A & \\ & 1 \end{pmatrix}
\]
A representation $\rho : \pi_1(\orb^2) \rightarrow \SL_\pm(n,\R)$ is a \textit{type preserving representation} if it maps orientation preserving elements in $\pi_1(\orb^2)$ to matrices of determinant 1 and orientation reversing elements to matrices of determinant $-1$. Note that the representation in $\Hom^{\operatorname{ref}}(\pi_1(\orb^2),\SL_\pm(n,\R))$ is a type preserving representation. 

\begin{theorem}[Theorem 3.6(2), \cite{Por23}]\label{Por23 Thm3.6(2)}

Let $\orb^2$ be a compact, non-orientable 2-orbifold with $\chi(\orb^2)<0$, $\rho : \pi_1(\orb^2) \rightarrow \SL_\pm(n,\R)$ be a type preserving representation, and $d_{tp}=\dim H^1(\orb^2,\R^n_\rho) \geq 0$. Assume that $\rho$ restricted to the orientation covering of $\orb^2$ is $\C$-irreducible. Then, for the type preserving embedding, a neighborhood of the character of $\rho$ in $\chi(\orb^2,\SL_\pm(n+1,\R))$ is homeomorphic to
\[
U \times \R^b \times \operatorname{Cone}(\Sphere^{d_{tp}-1} \times \Sphere^{d_{tp}-1})/\sim
\]
where $U$ is a neighborhood of the character of $\rho$, and $\sim$ is an equivalence of the antipodal map in $\Sphere^{d_{tp}-1} \times \Sphere^{d_{tp}-1} \subset \Sphere^{2d_{tp}-1} \subset \R^{2d_{tp}}$. Moreover, $\chi(\orb^2,\SL_\pm(n,\R))$ corresponds to $U \times \{0\} \times \{0\}$.

\end{theorem}

Assume that $\orb^2$ is $D^2(;n_1,n_2,n_3,n_4)$ where each corner-reflector has an order of at least 3. To apply the theorem above, we need to compute the first Betti number $b =\dim H^1(\pi_1(\orb^2),\R)$, and the specific quantity $d_{tp}$. First, we can compute the first Betti number by de Rham cohomology. Since $\orb^2$ is topologically a disk, $H^1(\pi_1(\orb),\R)$ is trivial. Therefore, $b=0$. To compute $d_{tp}$, we will use the below lemma.

\begin{lemma}[Lemma 4.6(b), \cite{Por23}]\label{Por23 Lemma4.6(b)}

Let $\orb^2$ be a compact non-orientable hyperbolic 2-orbifold. Let $\rho : \pi_1(\orb^2) \rightarrow \SL_\pm(3,\R)$ be the holonomy of a convex projective structure. Then
\[
d_{tp} = \dim (\operatorname{Teich} (\orb^2)) - f(\partial \orb^2)
\]
where $f(\partial \orb^2)$ is the number of full 1-orbifold boundary components of $\orb^2$.

\end{lemma}

Since $\orb^2$ has no boundary component, $f(\partial \orb^2)$ is zero. We can compute the dimension of Teichmüller space by Thurston's theorem.

\begin{theorem}[Corollary 13.3.7, \cite{Thu22}]\label{Thu22 Cor13.3.7}

Let $\orb$ be an orbifold with $\chi(\orb)<0$. Then,
\[
\dim (\operatorname{Teich} (\orb)) = -3\chi(|\orb|) + 2k + l
\]
where $k$ is the number of cone points, and $l$ is the number of corner reflections.
    
\end{theorem}

If we apply $\orb = D^2(;n_{12}, n_{23}, n_{34}, n_{14})$, then $d_{tp}=-3 +0 +4 =1$. Then, we should compute $\operatorname{Cone}(\Sphere^{0} \times \Sphere^{0})/\sim$. Recall that $\Sphere^0$ is the set of a pair of antipodal points. Thus, this space is homeomorphic to the real line $\R$. Furthermore, we already know the deformation space of quadrilateral whose edge orders are greater than or equal to 3, $U \cong \R^4$ by Corollary \ref{Cor of CG05}. Therefore, we can conclude that the character variety of $\SL_\pm(4,\R)$ is locally homeomorphic to $\R^5$. Our main result is consistent with Porti's theorems in terms of the dimension of the deformation space. Moreover, we can get a representation of $\SL_\pm(3,\R)$ when $a_4$ and $v_{44}$ are zero. From the proof in chapter \ref{ch:3.2}, we can consider that the product value $a_4 v_{44}$ is a variable in $\R$. Thus, we also checked the additional argument of Theorem \ref{Por23 Thm3.6(2)}. Thus, we conclude that the Figure \ref{Figure graph standard} represents $\operatorname{Cone}(\Sphere^{0} \times \Sphere^{0})/\sim$.

\subsection{Convex cocompactness}
\label{ch:4.2}

Danciger, Gu'eritand, and Kassel studied convex cocompactness, first introducing it in \cite{DGK18} for pseudo-Riemannian hyperbolic space. This concept was later extended to real projective geometry in \cite{DGK23}. The earlier work focused on the right-angled case, while Lee and Marquis examined the general case in \cite{LM19}. These studies were subsequently unified and further developed in \cite{DGKLM23}.

There exist three notions of projective convex cocompactness, of which we will focus on two. Let $H$ be an infinite discrete subgroup of $\GL(V)$. The group $H$ is said to be \textit{naively convex cocompact} in $\mathbb{P}(V)$ if it acts properly discontinuously on some properly convex open subset $\Omega \subset \mathbb{P}(V)$ and cocompactly on some nonempty $H$-invariant closed convex subset $\mathfrak{C} \subset \Omega$. Moreover, $H$ is \textit{convex cocompact} if it is naively convex cocompact and $\mathfrak{C}$ can be chosen ``large enough'' in the sense that its closure in $\mathbb{P}(V)$ contains all accumulation points of all possible $H$-orbits $H \cdot y$ with $y \in \Omega$.

To determine whether a representation $\rho$ induces a convex cocompact subset in $\mathbb{P}(V)$, we refer to Theorem 1.3 in \cite{DGKLM23}. Before stating the theorem, we recall some definitions related to Coxeter groups and the Cartan matrix. Let $\Gamma$ be a Coxeter group. A Coxeter group is \textit{irreducible} if it cannot be written as the direct product of two nontrivial standard subgroups. This is equivalent to stating that a Coxeter group has no order-2 elements. An irreducible Coxeter group is \textit{spherical} if it is finite and \textit{affine} if it is infinite and virtually abelian (i.e., it has an abelian subgroup of finite index). A \textit{standard subgroup} of $\Gamma$ is a subgroup of $\Gamma$ that is also a Coxeter group. A \textit{Cartan submatrix} is the Cartan matrix associated with a standard subgroup.

We introduce two conditions that serve as the assumptions for the following theorem. These conditions can be easily checked if a Coxeter diagram is provided. Suppose that $\Gamma$ is a Coxeter group.

\begin{itemize}
\item [$\neg(IC)$] There do not exist disjoint standard subgroups $\Gamma'$, $\Gamma''$ of $\Gamma$ such that both are infinite and commute.
\item [$(\Tilde{A})$] Let $\Gamma'$ be a standard subgroup of $\Gamma$. If $\Gamma'$ has a Coxeter diagram with $N \geq 3$ nodes and is irreducible and affine, then it must be of type $\Tilde{A}_{N-1}$.
\end{itemize}

\begin{figure}[h]
    \centering
    \includegraphics{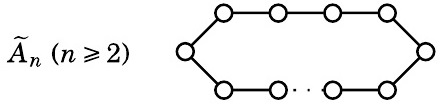} \\
    \includegraphics{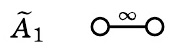}
    \caption{Coxeter diagrams of type $\tilde{A}_k$ (\cite{DGKLM23}, Appendix A)}
    \label{Coxeter diagrams of Ak}
\end{figure}

The following theorem states that if a representation originates from an infinite Coxeter group satisfying these two conditions, then convex cocompactness can be determined using the following equivalences.

\begin{theorem}[Theorem 1.8, \cite{DGKLM23}]
\label{DGKLM Thm1.8}

Let $\Gamma$ be an infinite Coxeter group satisfying conditions $\neg(IC)$ and $(\Tilde{A})$. Let $\rho \in \Hom^{\operatorname{ref}}(\Gamma, \GL(V))$ and $M$ be the Cartan matrix. Then, the following statements are equivalent:

\begin{itemize}
\item [(1)] $\rho(\Gamma)$ is naively convex cocompact in $\mathbb{P}(V)$.
\item [(2)] $\rho(\Gamma)$ is convex cocompact in $\mathbb{P}(V)$.
\item [(3)] For any nonempty irreducible standard subgroup $\Gamma'$ of $\Gamma$, the Cartan submatrix $M'$ is not of zero type, meaning its lowest eigenvalue is nonzero.
\item [(4)] For any standard subgroup $\Gamma'$ of $\Gamma$ of type $\Tilde{A}_k$ with $k \geq 1$, the determinant of the corresponding submatrix is nonzero.
\end{itemize}

\end{theorem}

We aim to apply Theorem \ref{DGKLM Thm1.8} to our main result, Theorem \ref{Main Theorem}. The Coxeter diagram associated with this theorem is shown in Figure \ref{Coxeter diagram of Thm1}. A bold line of the Coxeter diagram implies an infinite edge order.

\begin{figure}[h]
    \centering
    \includegraphics[height=5cm]{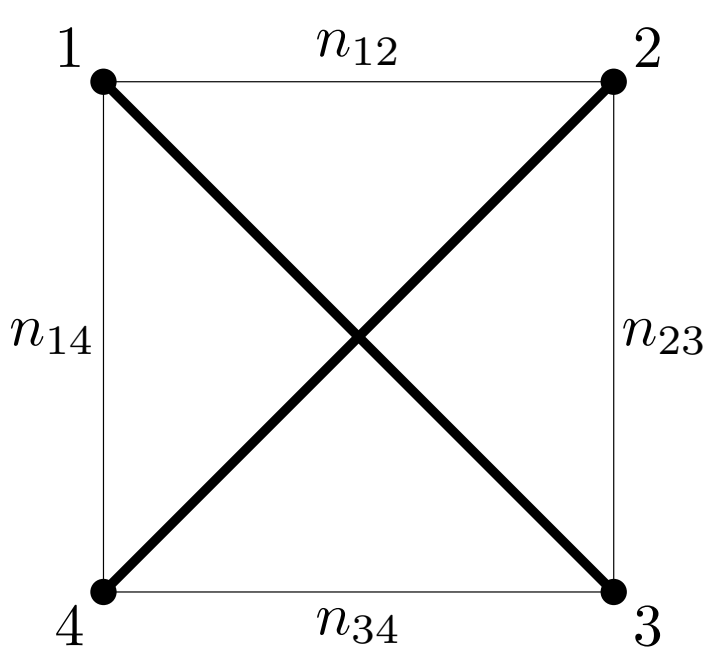}
    \caption{The Coxeter diagram of Theorem 1}
    \label{Coxeter diagram of Thm1}
\end{figure}

\begin{proposition} \label{Additional prop}

Let $\hat{P}$ be a convex Coxeter orbifold $D^2(; n_1, n_2, n_3, n_4) \times \R$, where $n_1, n_2, n_3, n_4$ are positive integers greater than or equal to 3. Let $\Gamma$ be the Coxeter group related to $\hat{P}$ and let $\rho \in \Hom^{\operatorname{ref}}(\Gamma,\SL_\pm(4,\R))$. Then,

\begin{itemize}

\item $\rho(\Gamma)$ is convex cocompact if and only if $T_{13}>4$ and $T_{24}>4$.
\item $\rho(\Gamma)$ corresponding to the concurrent case is always convex cocompact.
\item The boundary of $\Defss(\hat{P})$ lies entirely within one of the two connected components of $\Defssgen(\hat{P})$.

\end{itemize}

\end{proposition}

\begin{proof}

The first statement follows directly from Theorem \ref{DGKLM Thm1.8}. To apply the theorem, we should verify that $\Gamma$ satisfies $\neg(IC)$ and $(\Tilde{A})$. First, verify $\neg(IC)$ part. The only possible pair of disjoint infinite standard subgroups is given by $\Gamma_1$, generated by $R_1$ and $R_3$, and $\Gamma_2$, generated by $R_2$ and $R_4$. However, these two subgroups do not commute, confirming that $\neg(IC)$ holds. Next, we verify condition $(\Tilde{A})$. Every Coxeter diagram of a standard subgroup of $\Gamma$ with $N = 3$ forms a triangular graph labeled with two finite integers and one infinity, which is not affine. Thus, there is no irreducible affine standard subgroup of $\Gamma$ with $N \geq 3$, confirming $(\Tilde{A})$. Consider statement (4) of Theorem \ref{DGKLM Thm1.8}. For our orbifold $\hat{P}$, there is no such $\Gamma'$ of type $\Tilde{A}_k$ with $k>1$. Suppose that $k=1$. This simply states that $M_{ij}M_{ji} > 4$ for all $i \neq j$ with $n_{ij} = \infty$. Therefore, $\rho(\Gamma)$ is convex cocompact in $\mathbb{P}(V)$ if and only if both $M_{13}M_{31}$ and $M_{24}M_{42}$ are greater than 4.

Next, prove the second statement. Note that the concurrent case admits the structure of quadrilateral $D^2(n_1,n_2,n_3,n_4)$, and this always admits a hyperbolic structure. Therefore, every concurrent case implies that $T_{13}>4$ and $T_{24}>4$. (We also verified in chapter \ref{ch:3.1} with the parametrization.) By the first statement, the second statement is proved.

Lastly, prove the third statement. Since
\[
\{v_{23}<0\} \times \{v_{24}<0\} \times \{v_{34}<0\} \times \{T_{13}\geq4\} \times \{T_{24}\geq4\},
\]
the boundary of $\Defss(\hat{P})$ implies that $T_{13}=4$ or $T_{24}=4$. As we saw in Figure \ref{Figure graph standard}, $\Defssgen(\hat{P})$ is disconnected by $\Defsscon(\hat{P})$ in $\Defss(\hat{P})$. This proves the third statement.

\end{proof}

\section[Appendix]{Appendix}
\label{ch:5}
\subsection{The calculations of cyclic invariants}
\label{ch:5.1}

We can reinterpret the proof in Chapter \ref{ch:3.2} using cyclic invariants. A \textit{cyclic invariant} is $M_{i_1 i_2}M_{i_2 i_3} \cdots M_{i_k i_1}$ for distinct indices $i_1, \cdots, i_k$, where $M$ is the Cartan matrix. (Some define a cyclic invariant in the form of $\log \frac{M_{i_1 i_2}M_{i_2 i_3} \cdots M_{i_k i_1}}{M_{i_1 i_k}M_{i_k i_{k-1}} \cdots M_{i_2 i_1}}$, as the cyclic invariant with inverse indices can be uniquely determined.) Note that cyclic invariants of length 2 are the invariant values in the fourth condition of Vinberg's conditions. The following proposition explains why this is called an invariant.

\begin{proposition}[Proposition 16, \cite{Vin71}]\label{Vin71 Prop16} Let $M$ and $M'$ be two Cartan matrices. Then, the following are equivalent: \begin{itemize} \item Their cyclic invariants are identical. \item $M' = DMD^{-1}$ for some diagonal matrix $D$ with positive diagonal elements. \item $M$ and $M'$ are the Cartan matrices of the same Coxeter polytope, up to isomorphism. \end{itemize} \end{proposition}

By Proposition \ref{Vin71 Prop16}, two Cartan matrices with the same cyclic invariants imply they are the same point in $\Defss(\hat{P})$. Furthermore, we don't need to check all cyclic invariants of the given Cartan matrix. For $D^2(;n_1,n_2,n_3,n_4) \times \mathbb{R}$, two cyclic invariants of length 2 and three cyclic invariants of length 3 determine all cyclic invariants. This can be verified by its Coxeter graph in Figure \ref{Coxeter diagram of Thm1}. Denote $(i_1, \cdots, i_k)$ as the cyclic invariant with the given index order. Then, we can select five cyclic orders that determine all cyclic invariants, each of which has a one-to-one correspondence with our parametrization of the solution space in Chapter \ref{ch:3.2}
\[
    (1,3) = M_{13}M_{31}, \quad
    (2,4) = M_{24}M_{42}, \quad
    (1,2,3) = \mu_{12} v_{23}, \quad
    (1,2,4) = \mu_{12} v_{24}, \quad
    (1,4,3) = \frac{\mu_{34}}{v_{34}}.
\]
Therefore, we can also prove Theorem \ref{Main Theorem} using cyclic invariants. The preservation of the cyclic invariant is currently only achieved in semisimple cases, so this assumption must be made.

\begin{remark} Five cyclic invariants can determine all cyclic invariants of Theorem \ref{Main Theorem}. \end{remark}

Without loss of generality, (1,3), (2,4), (1,2,3), (1,2,4), (1,3,4) determine all other cyclic invariants.
\begin{align*}
(1,3,2) &= \frac{\mu_{12}\mu_{23}(1,3)}{(1,2,3)} \\
(1,4,2) &= \frac{\mu_{12}\mu_{14}(2,4)}{(1,2,4)} \\
(1,4,3) &= \frac{\mu_{14}\mu_{34}(1,3)}{(1,3,4)} \\
(2,3,4) &= \frac{(2,4)(1,2,3)(1,3,4)}{(1,3)(1,2,4)} \\
(2,4,3) &= \frac{\mu_{23}\mu_{34}(1,3)(1,2,4)}{(1,2,3)(1,3,4)} \\
(1,2,3,4) &= \frac{(1,2,3)(1,3,4)}{(1,3)} \\
(1,2,4,3) &= \frac{\mu_{34}(1,3)(1,2,4)}{(1,3,4)} \\
(1,3,2,4) &= \frac{\mu_{23}(1,3)(1,2,4)}{(1,2,3)} \\
(1,3,4,2) &= \frac{\mu_{12}\mu_{24}(1,3,4)}{(1,2,4)} \\
(1,4,2,3) &= \frac{\mu_{14}(2,4)(1,2,3)}{(1,2,4)} \\
(1,4,3,2) &= \frac{\mu_{12}\mu_{23}\mu_{34}\mu_{14}(1,3)}{(1,2,3)(1,3,4)}
\end{align*}

\begin{remark} A cyclic invariant with the same indices but in a different order can be determined by the other one. \end{remark}

This fact can be easily shown by multiplying those two cyclic invariants. From this, we can conclude that the cyclic invariant with a different order is fully determined by the original one.

\begin{remark} From the Coxeter graph, each closed loop implies a cyclic invariant. Moreover, if the loop can be divided into the sum of smaller loops, then its cyclic invariants can be determined by those two cyclic invariants. \end{remark}

For example, $(1,2,3,4)$ is a quadrilateral, which can be expressed as the sum of two triangles: $(1,2,3)$ and $(1,3,4)$, as shown in Figure \ref{Coxeter diagram of Thm1}.

\subsection{Notebooks of Mathematica} \label{ch:5.2}

You can access the Mathematica notebooks on my website.

\href{https://sites.google.com/view/jaesungbae/home}{https://sites.google.com/view/jaesungbae/home}

\subsubsection{Representations of the general position case} \label{ch:5.2.1}

We assume the same condition as stated in \ref{ch:3.3}, where \( \alpha_j = e^*_j \) for all \( j=1,2,3,4 \), and \( v_1 = [2,-1,-1,-1]^T \).
Recall that the reflections are given by  
\[
R_j(x) = x - \alpha_j(x) v_j.
\]
We then verify that for every finite order $n_{ij}$, the composition \( (R_i R_j)^{n_{ij}} \) becomes the identity matrix.
This Mathematica note confirms that the reflections we assumed indeed generate the desired Coxeter group.

\subsubsection{Representations of the concurrent case}

We assume the same conditions as stated in \ref{ch:3.1}, i.e, $\alpha_j = e^*_j$, for $j=1,2,3$ and $\alpha_4 = e^*_1 - e^*_2 + e^*_3$, so that it corresponds to the concurrent case. Similar to \ref{ch:5.2.1}, verify that for every finite order $n_{ij}$, the composition \( (R_i R_j)^{n_{ij}} \) becomes the identity matrix so that the reflections generate the desired Coxeter group.

\subsubsection{Determinant of the Cartan matrix in Chapter \ref{ch:3.2}}

This Mathematica notebook demonstrates that the determinant of the Cartan matrix in Chapter \ref{ch:3.2} cannot be zero if we assume that either $T_{13}$ or $T_{24}$ is equal to 4. This leads to the fact that the solution space, assuming $\alpha_4 \neq 0$ and $v_{44} = 0$, excludes points with the above assumption.

\subsubsection{Analysis of the cone space of Porti's theorem}

This Mathematica notebook shows the minimum and maximum values of $a_4 v_{44}$ from Chapter \ref{ch:3.2}, by changing the finite orders $n_{12}, n_{23}, n_{34}, n_{14}$ and the values of the infinite orders $T_{13}$ and $T_{24}$. From this notebook, we can check that the variable of the cone space of Theorem \ref{Por23 Thm3.6(2)}, which is $a_4 v_{44}$, has restrictions for each $n_{12}, n_{23}, n_{34}, n_{14}, T_{13}$, and $T_{24}$.

For example, if we assume that all finite orders are 3 and $T_{13} = T_{24} = 6$, then the minimum of $a_4 v_{44}$ is almost close to 2. This implies that there is no concurrent case with those assumptions. If we assume $T_{13} = T_{24} = 16$, the minimum of $a_4 v_{44}$ becomes 0. In fact, if we compute $T_{13} T_{24}$ in the concurrent case, the minimum of $T_{13} T_{24}$ is 256 if we assume all finite edge orders are 3. If we consider finite edge orders greater than 3, this minimum grows larger.


\nocite{*}
\bibliographystyle{alpha}
\bibliography{references}
\end{document}